\documentclass[12pt]{article}

\usepackage{amsmath,amsfonts}
\usepackage{graphicx}
\usepackage{color}

\def\G{{\cal G}}
\def\({\left(}
\def\){\right)}

\newif\ifblog
\newif\iftex
\blogfalse
\textrue

\newcommand{\e}{\varepsilon}

\def\indi{{\bf 1}}

\def\P{{\mathbb P}}
\def\E{{\mathbb E}}

\def\R{{\mathbb R}}

\def\K{{\mathcal K}}

\def\R{{\mathbb R}}

\def\F{{\mathcal F}}

\def\p{{\partial}}

\def\M{{\mathcal{M}}}
\newcommand{\al}{\alpha}

\newcommand{\disp}{\displaystyle}
\newcommand{\nd}{\noindent}
\newcommand{\cd}{(\cdot)}
\newcommand{\wdt}{\widetilde}
\newcommand{\la}{\lambda}

\newcommand{\bea}{\bed\begin{array}{rl}}
\newcommand{\eea}{\end{array}\eed}
\newcommand{\ad}{&\!\!\!\disp}
\newcommand{\aad}{&\disp}
\newcommand{\barray}{\begin{array}{ll}}
\newcommand{\earray}{\end{array}}

\newcommand{\beq}[1]{\begin{equation} \label{#1}}
\newcommand{\eeq}{\end{equation}}
\newcommand{\bed}{\begin{displaymath}}
\newcommand{\eed}{\end{displaymath}}

\newcommand{\bedd}{\bed\begin{array}{l}}
	\newcommand{\eedd}{\end{array}\eed}

\newtheorem{theorem}{Theorem}[section]
\newtheorem{lemma}[theorem]{Lemma}
\newtheorem{definition}[theorem]{Definition}

\newtheorem{proposition}[theorem]{Proposition}

\newtheorem{remark}[theorem]{Remark}

\newenvironment{proof}{\noindent {\sc Proof:}}{\strut\hfill $\Box$} 

\newcommand{\LL}{{\cal L}}
\newcommand{\wt}{\widetilde}
\newcommand{\RR}{{\mathbb R}}

\makeatletter
\@addtoreset{equation}{section}

\makeatother

\title{Properties of Switching Jump Diffusions:
Maximum Principles and Harnack Inequalities}

\author{Xiaoshan Chen,\thanks{School of Mathematical Sciences, South China Normal University, Guangdong, China,
xschen@m.scnu.edu.cn. This research was partially supported by NNSF of China (No.11601163) and NSF of Guangdong Province of China (No. 2016A030313448).} \and Zhen-Qing Chen,\thanks{Departments of Mathematics,
University of Washington,
Seattle, WA 98195, USA, zqchen@uw.edu. This
research was partially supported
by NSF grant DMS-1206276.
} \and Ky Tran,\thanks{Department
of Mathematics, College of Education, Hue University, Hue city, Vietnam.
quankysp@gmail.com. This research was
partially supported by Vietnam National Foundation for Science and Technology Development (NAFOSTED) under grant 101.03-2017.23.}
\and George Yin\thanks{Department
of Mathematics, Wayne State University, Detroit, MI 48202 USA,
gyin@math.wayne.edu. This research was
partially supported by the National Science Foundation under grant DMS-1710827.}
}

\begin{document}
\maketitle

\begin{abstract}
This work examines a class of switching jump diffusion processes.
The main effort is devoted to proving the maximum principle and obtaining the Harnack inequalities.
Compared with the diffusions and switching diffusions, the associated operators for switching jump diffusions
are non-local, resulting in more difficulty in treating such systems.
Our study is carried out by taking into consideration of the interplay of
 stochastic processes and the associated systems of
integro-differential equations.

\bigskip
\nd {\bf Keywords.} jump diffusion, regime switching, maximum principle, Harnack inequality.

\bigskip
\nd {\bf Mathematics Subject Classification.} 60J60, 60J75, 35B50, 45K05.

\bigskip
\nd {\bf Brief Title.} Maximum Principles and Harnack Inequalities

\end{abstract}

\newpage

\setlength{\baselineskip}{0.25in}
\section{Introduction}
In recent years,
 many different fields require the handling of
 dynamic systems in which there is a component representing random environment and other factors that are not
given as a solution of the usual differential equations.
Such systems have drawn new as well as resurgent attention because of the urgent needs of systems modeling, analysis,
and optimization in a wide variety of applications. Not only do the applications arise from the traditional fields of mathematical modeling,
 but
also they have appeared in emerging application areas such as  wireless communications, networked systems,
autonomous systems, multi-agent systems, flexible manufacturing systems,
 financial engineering, and
biological and
 ecological systems, among others.
Much effort has been devoted to the so-called hybrid systems. Taking randomness into consideration, a class of
such systems known as switching diffusions  has been investigated thoroughly; see for example, \cite{MY06,YZ10} and references therein.
Continuing our investigation on regime-switching systems,
this paper focuses on
a class switching jump diffusion processes.
To work on such systems, it is necessary to study
a number of fundamental properties.
Although we have a good understanding of switching diffusions, switching jump diffusions are more difficult to deal with.
One of the main difficulties is the operator being non-local.  When we study switching diffusions, it has been demonstrated that
although they are similar to
diffusion processes, switching diffusions have some  distinct features.
With the non-local operator used, the distinctions are even more pronounced.
Our primary motivation stems from the study of  a family of Markov processes in
which continuous dynamics, jump discontinuity, and discrete events coexist.
Their interactions reflect the salient features of the underlying systems.
Specifically,  we focus on regime-switching jump diffusion processes, in which the switching process is not exogenous but depends on the jump
diffusions. The distinct features of the systems include the presence of non-local operators, the coupled systems
of equations, and the tangled information due to the dependence of the switching process on the jump diffusions.

To elaborate a little more on the systems,  similar to \cite[Section 1.3, pp. 4-5]{YZ10}, we begin with the following description.
Consider a two component process $(X_t,\Lambda_t)$, where $\Lambda_t \in \{1,2\}$. We call $\Lambda_t$ the discrete event process with state space
$\{1,2\}$.
Imagine that we have two parallel  planes. Initially, $\Lambda_0=1$. It then sojourns in the state $1$ for a random duration.
During this period, the diffusion with jump traces out a curve on  plane 1 specified by the drift, diffusion, and jump coefficients. Then a random
switching takes place at a random time $\tau_1$, and $\Lambda$ switches to plane 2 and sojourns there for a random duration.
During this period, the diffusion with jump traces out a curve on plane 2 with different drift, diffusion, and jump coefficients.
What we are interested in is the case that $\Lambda_t$ itself is not Markov, but only the
two-component process $(X_t,\Lambda_t)$ is a Markov process.
Treating such systems,
similar to the study of switching diffusions,
 we may consider a number of questions:
Under what conditions,
will the processes be
recurrent and positive recurrent? Under what conditions, will the process be positive recurrent? Is it true that positive recurrence
implies the existence of an ergodic measure. To answer these questions, we need to examine a number of issues of the
 switching jump diffusions and the associated systems of integro-partial differential equations.

Switching jump diffusions models arise naturally in many applications.
To illustrate,
consider the following
motivational example--an optimal stopping problem.
  It is
 an extension of the optimal stopping problem
 for switching diffusions with diffusion dependent switching in \cite{Liu16}. We assume
 that the dynamics are described by switching jump diffusions rather than switching diffusions.
 Consider a two component Markov process $(X_t,\Lambda_t)$ given by
 $$ dX_t = b(X_t, \Lambda_t) dt +\sigma (X_t, \Lambda_t) dW(t) +  \int _{\RR_0} c (X_{t-},\Lambda_{t-},z) \wdt N_0(dt, dz),$$
 where $b\cd$, $\sigma\cd$, and $c \cd$ are suitable real-valued functions, $\wdt N_0\cd$ is a compensated real-valued Poisson process,
 $W\cd$ is a real-valued Brownian motion, and $\RR_0= \RR - \{0\}$.
 Because the example is for motivation only, we defer
 the discussion of the precise setup, formulation, and  conditions  needed
 for switching jump diffusions to the next section.
 We assume that $\Lambda$ depends on the dynamics of $X$.
 Denote the filtration by
 $\{ \F_t\}_{t\ge 0}$ and let
 ${\cal T}$ be the collection of $\F_t$-stopping times.
 Then the treatment of the optimal stopping problem leads to the consideration of the following
 value function
 $$V(x,i)= \sup_{\wdt \tau \in {\cal T}} \E_{x,i}  \Big[\int^{\wdt \tau}_0 [ e^{-\beta  t} L(X_t,\Lambda_t) dt +
 e^{-\beta \wdt \tau}
 \wdt G (X_{\wdt \tau}, \Lambda_{\wdt \tau}) ]\Big],$$
 where $L\cd$ and $\wdt G\cd$ are suitable functions, and $X_0=x$ and $\Lambda_0=i$.
 As an even more specific example, consider an asset model
   $$ dX_t = b( X_t, \Lambda_t) dt +\sigma (  X_t, \Lambda_t) dW(t) +  \int _{\RR_0} c (\Lambda_{t-},z) X_{t-} \wdt N_0(dt, dz).$$
   Then the risk-neutral price of the
 perpetual American put option is given by
  $$V(x,i)= \sup_{\wdt \tau \in {\cal T}} \E_{x,i}[K- X_{\wdt \tau}]^+.$$

  One of the
  motivations for using jump-diffusion type models is that it has been observed  empirically that
distributions of the returns often  have heavier tails than that of normal distributions.
In particular, if we take $N(t)$ to be
a one-dimensional stationary Poisson process with
$\E N(t)= \lambda t$ for some $\lambda>0$, and take
the compensated Poisson process to be
$\widetilde N(t)= N(t)-\lambda t$. The resulted system is used widely in
option
pricing and mean-variance portfolio selections.

Next, consider a modification of a frequently used
 system in control theory.
Let $\Gamma$ be a compact subset of $\R^d -\{0\}$ that is the
range space of the impulsive jumps. For any subset $B$ in $\Gamma$,
$N(t,B)$ counts the number of impulses on $[0,t]$ with values in
$B$. Consider
$$\begin{array}{ll}
&\disp dX_t=b(X_t,\Lambda_t)dt+\sigma(X_t,\Lambda_t)dW_t+dJ_t, \\
&\disp J_t=\int_0^t\int_\Gamma c(X_{s-}),\Lambda_{s-}),\gamma) N(ds,d\gamma),
\end{array}
$$
with $X_0=x, \Lambda_0=\Lambda$,
 together with
a transition probability specification of the form
$$
\P\{\Lambda_{t+\Delta t}=j|\Lambda_t=i,
(X_s,\Lambda_s),s\le t \}= q_{ij}(X_t)\Delta t + o(\Delta t), \quad
i\not =j,
$$
where $b$ and $\sigma$ are suitable vector-valued and matrix-valued functions, respectively, and $W$ is a standard vector-valued Brownian motion.
Assume that
$N(\cdot,\cdot)$ is independent of the
Brownian motion $W\cd$ and the switching process $\Lambda\cd$.
Alternatively, we can write
$$d\Lambda_t=\int_\R h(X_{s-}),\Lambda_{s-}),z)N_1(dt,dz),$$
where $h(x,i,z)=\sum_{j\in\M}(j-i)\indi_{\{z\in\Delta_{ij}(x)\}}$ with $\Delta_{ij}(x)$ being the consecutive left closed and right open intervals of the real line, and $\widetilde N(t,B)$ being a compensated Poisson measure, which is independent of the Brownian motion $W(t)$, $0<\lambda<\infty$ is known as the jump rate and $\pi(B)$ is the jump measure;
$N_1(dt,dz)$ is a Poisson measure with intensity $dt\times m_1(dz)$, and $m_1(dz)$ is the Lebesgue measure on $\R$, $N_1(dt,dz)$ is independent of the Brownian motion $W(t)$ and the Poisson measure $\widetilde N(\cdot,\cdot)$.
Define a
compensated or centered Poisson measure as
$$\wdt N(t,B)=N(t,B)- \lambda t \pi(B), \ \hbox{ for } \ B \subset
\Gamma,$$
where $0<\la<\infty$ is known as the jump rate and $\pi\cd$ is the
jump distribution (a probability measure).
In the above, we used the
setup similar to \cite[p. 37]{Kushner90}. With this centered Poisson
measure, we can rewrite $J_t$ as $$ J_t= \int^t_0
\int_\Gamma g(X_{s-},\Lambda_{s-},\gamma) \wdt N(ds,d\gamma) + \la \int^t_0
\int_\Gamma g(X_{s-},\Lambda_{s-},\gamma) \pi(d\gamma) ds.$$
The related jump diffusion models without switching have been
used in a wide range of applications in control systems; see \cite{Kushner90}
and references therein.

 We devote our attention to the maximum principle and Harnack inequalities for the jump-diffusion processes with regime-switching in this paper.
Apart from being interesting in their own right,  they play very important roles in analyzing many properties such as recurrence, positive recurrence,
and ergodicity of the
 underlying systems.
 There is
  growing interest in treating switching jump systems; see \cite{Xi09} and many references therein.
 However, up to date,
 there seems to be no results on maximum principles and
    Harnack inequality for jump-diffusion processes with regime switching.
    As was alluded to in the previous paragraph,
 the main difficulty is that the operators involved are non-local. Thus, the results obtained for the systems
 (known as weakly coupled elliptic systems)
  corresponding to switching diffusions cannot be
 carried over.  Thus new approaches and ideas have to be used.

Looking into the literature,
in \cite{Evans10}, Evans proved the maximum principle for uniformly elliptic equations.
In the classical book \cite{PW67},  Protter and  Weinberger treated maximum principle for  elliptic equations as well as Harnack inequalities
and  generalized maximum principle together with a number of other topics.
For switching diffusion processes,  several papers studied Harnack inequality for the weakly coupled systems of elliptic equations. In
 \cite{CZ97},
 Chen and Zhao assumed H\"older continuous coefficients, and carried out the proofs based on the representations and estimates of the Green function and harmonic measures of the operators in small balls. In \cite{AGM99}, Arapostathis, Ghosh, and Marcus assumed only measurability of the coefficients to prove the desired results; their proofs were based on the approach of Krylov \cite{Kry87} for estimating the oscillation of a harmonic function on bounded sets.
There have been much interest in treating jump processes and associated non-local operators.
In a series of papers,
  Bass and Kassmann \cite{BK05},
 Bass, Kassmann,  and Kumagai \cite{BKK},  Bass and Levin \cite{BL02},
 Chen and Kumagai \cite{CK1, CK2, CK3}, Song and Vondracek \cite{SV05}
examined Harnack inequalities for
Markov processes with discontinuous sample paths.
In \cite{CS09}, Caffarelli and  Silvestre considered nonlinear integro-differential equations arising from
stochastic control problems with pure jump L\'evy processes (without a Brownian motion) using a purely analytic approach.  Nonlocal
version of ABP (Alexandrov-Bakelman-Pucci) estimate, Harnack inequality, and regularity were obtained.
Most recently, Harnack inequality for solutions to the
Schr\"odinger operator were dealt with
in  \cite{AR16} by Athreya and Ramachandran
for jump diffusions whose associate operator is
an integro-differential operator includes
the pure jump part as well as elliptic part.
Their approach is based on the comparability of Green functions and Poisson kernels using conditional gauge function
and strong regularity is assumed on the coefficients of the diffusion and jumping components.

In this paper, we focus on stochastic processes that have a switching component in addition to the jump diffusion component.
The switching in fact is ``jump diffusion dependent'';
more precise notion will be given in the formulation section. When the switching component is missing,  it reduces to the jump diffusion processes; when the continuous disturbance due to Brownian motion is also missing, it reduces to the case of pure jump processes. If only the jump process is missing, it reduces to the case of switching diffusions.
Compared to the case of switching diffusion processes, in lieu of systems of elliptic partial differential equations, we have to deal with  systems of integro-differential equations. Using mainly a probabilistic approach, we establish the maximum principles.
 Because local analysis alone is not adequate,
 the approach treating  Harnack inequality for switching diffusion processes cannot be used in the current case.
We adopt the probabilistic approach via Krylov type estimates from \cite{BL02}, which was further extended in
   \cite{BK05,  Foondun09, SV05},
  to derive the Harnack inequality for the nonnegative solution of the system of integro-differential equations.

The rest of the paper is arranged as follows. Section \ref{sec:formulation} presents the formulation of the problem. In Section \ref{sec:maximum}, we develop the maximum principle for
 regime-switching jump-diffusions processes,
 using a probabilistic approach that allows us to work
  under a quite general context.
    We obtain the Harnack inequality for the
    regime-switching
    jump-diffusions processes in Section \ref{sec:harnack}.
 Finally, the paper is concluded with further remarks.

\section{Formulation}\label{sec:formulation}

Throughout the paper,
 we use $z'$ to denote the transpose of $z\in \R^{l_1
	\times l_2}$ with $l_1, l_2 \ge 1$, and $\R^{d\times 1}$ is simply
written as $\R^d$. If $x\in \R^d$, the norm of $x$ is denoted by
$|x|$.
For $x_0\in \R^d$ and $r>0$, $B(x_0, r)$ denotes the open ball in $\R^d$ centered at $x_0$ with radius $r>0$.
If $D$ is a Borel set in $\R^d$, $\overline{D}$ and
$D^c =\R^d \setminus D$
denote the closure
and the complement of $D$, respectively. The space $C^2(D)$ refers to the class of functions whose partial derivatives up to order 2 exist and are continuous in $D$, and $C^2_b(D)$ is the subspace of $C^2(D)$ consisting of those functions whose partial derivatives up to order 2 are bounded. The indicator function of a set $A$ is denoted $\indi_A$.
Let $Y_t=(X_t,\Lambda_t)$ be a two component Markov process such that $X$ is an $\R^d$-valued process, and $\Lambda$ is a switching process taking values in a finite set $\mathcal M=\{1,2,\dots,m\}$.
Let $b(\cdot,\cdot): \R^d\times \mathcal M\mapsto \R^d$, $\sigma(\cdot,\cdot): \R^d\times\mathcal M\mapsto\R^d\times\R^d$, and
 for each $x\in \R^d$, $\pi_i (x, dz)$ is a $\sigma$-finite measure on $\R^d$  satisfying
$$
\int_{\R^d} (1\wedge |z|^2) \pi_i (x, dz) <\infty.
$$
Let $Q(x)=(q_{ij}(x))$ be an $m\times m$ matrix depending on $x$ such that $$q_{ij}(x)\geq 0 \quad \text{for } i\ne j, \quad  \sum_{j\in\mathcal{M}}q_{ij}(x)\le 0.$$ Define
$$Q(x)f(x,\cdot)(i):=\sum_{j\in\mathcal{M}}q_{ij}(x)f(x,j).$$
The generator $\mathcal G$ of the process $(X_t,\Lambda_t)$ is given as follows. For a function $f:\RR^d\times \M\mapsto \RR$ and $f(\cdot, i)\in C^2(\R^d)$
for each $i\in\mathcal M$, define
\begin{equation}\label{E:1}
\mathcal Gf(x,i)
= \LL_i f(x,i)+ Q(x)f(x, \cdot)(i), \quad   (x, i)\in \R^d\times \M,\end{equation}
where
\begin{eqnarray}
 \mathcal{L}_i f(x,i) &=&\sum^d_{k,l=1}a_{kl}(x,i){\frac{\p^2f(x,i)}{\p x_k\p x_l}+\sum^d_{k=1}b_{k}(x,i)\frac{\p f(x,i)}{\p x_k}}
\nonumber\\
 && +\int_{\R^d}  \left( f(x+z,i)-f(x,i)-\nabla f(x,i)\cdot z \indi_{\{|z|<1\}}\right)
  \pi_i(x, dz) ,    \label{e:2.2}
\end{eqnarray}
where
$a(x,i)=\sigma(x,i)\sigma'(x,i)$, $\nabla f(\cdot,i)$
denotes the gradient  of $f(\cdot,i)$.

Let $\Omega=D\big([ 0, \infty ), \R^d\times \M \big)$
 denote the space of all right continuous functions mapping $[0, \infty)$ to $\R^d \times \M$, having finite left limits. Define $(X_t, \Lambda_t)=w(t)$ for  $w\in \Omega$
 and let $\{\mathcal{F}_t\}$ be the right continuous filtration generated by the process $(X_t, \Lambda_t)$.
A probability measure $\P_{x,i }$ on $\Omega$ is a solution to the martingale problem for $\(\mathcal{G}, C^2_b(\R^d)\)$ started at $(x, i)$ if
\begin{itemize}
	\item[(a)] $\P_{x, i}( X_0=x, \Lambda_0=i )=1$,
	\item[(b)] if $f(\cdot, i)\in C^2_b(\R^d)$ for each $i\in \M$, then
	$$f(X_t, \Lambda_t)-f(X_0, \Lambda_0)-\int_0^t \mathcal{G} f( X_s, \Lambda_s )ds,$$
	is a $\P_{x, i}$ martingale.
	
	If for each $(x, i)$, there is only one such $\P_{x, i}$, we say that the martingale problem for
	$\(\mathcal{G}, C^2_b(\R^d)\)$ is well-posed.
\end{itemize}

\begin{definition}\label{D:har} {\rm
Let $U=D\times \M$
with $D\subset\mathbb{R}^d$
being a bounded connected open set. A bounded and Borel measurable function $f:\mathbb{R}^d\times\M\mapsto\mathbb{R}^d$ is said to be $\mathcal G$-harmonic in $U$ if for
any relatively compact open subset $V$ of $U$,
$$ f(x,i)=\mathbb \mathbb{E}_{x,i}\big[f\(X(\tau_V),\Lambda({\tau_V})\)\big] \quad \text{ for all}\quad (x, i)\in V,$$
where $\tau_V=\inf\{t\ge 0: \(X(t), \Lambda(t)\) \notin V \}$ is the first exit time
from $V$.
}\end{definition}

 Throughout the paper,   we assume conditions (A1)-(A3) hold
until further notice.

 \begin{itemize}
		\item[{(A1)}]  The functions   $\sigma
		(\cdot, i)$ and $b
		(\cdot, i)$ are bounded and continuous, $q_{ij}\cd$ is bounded and Borel measurable.

		\item[{(A2)}] There exists a constant $\kappa_0\in (0, 1]$ such that
		$$ \kappa_0|\xi|^2\le \xi'a(x, i)\xi \le
		\kappa_0^{-1}|\xi|^2 \quad  \text{ for all }\xi\in \R^d, x\in \R^d, i\in \M.
		$$

		\item[{(A3)}]
		There exists a $\sigma$-finite measure $\Pi (dz)$ so that   $\pi_i (x, dz) \leq \Pi (dz)$
for every $x\in \R^d$ and $i\in \M$
		and  		
		$$
		 \int_{\R^d} \left(1\wedge |z|^2\right) \Pi(dz) \le \kappa_1 <\infty.
		$$
		\item [(A4)] For any $i\in \M$ and $x\in \R^d$, $\pi_i (x, dz)=\wdt  \pi_i (x, z) dz$. Moreover,  for
		any $r\in (0,1]$,
		any $x_0\in \R^d$, any $x, y \in B(x_0, r/2)$ and $z\in B(x_0, r)^c$, we have
		$$
		\wdt\pi_i(x,z-x)\leq \al_r\wdt\pi_i(y,z-y),
		$$
		where $\alpha_r$ satisfies
		$1\leq \alpha_r \leq \kappa_2 r^{-\beta}$
		with $\kappa_2$ and $\beta$ being positive constants.
	\end{itemize}

\begin{remark}\label{R:2.2} {\rm
	(a) Under Assumptions (A1)-(A3), for each $i\in \M$, the martingale problem for $(\LL_i, C^2_b(\R^d))$
is well-posed for every starting point $x\in \R^d$ (see \cite[Theorem 5.2]{Komatsu}).
Then the switched Markov process $\(X_t, \Lambda_t\)$ can be constructed from
jump diffusions having infinitesimal generators $\LL_i$, $1\leq i\leq m$,
		as follows. 	
			Let $X^i$ be
		the strong Markov process whose distribution is the unique solution to the martingale problem
		$(\LL_i, C^2_b(\R^d))$.
			Suppose we start the process at $(x_0, i_0)$, run a subprocess $\wt X^{i_0}$ of $X^{i_0}$ that got killed with rate
			$-q_{i_0i_0} (x)$; that is, via Feynman-Kac transform $\exp \left(\int_0^t q_{i_0i_0} (X^{i_0}_s)ds \right)$. Note that this subprocess $\wt X^{i_0}$ has infinitesimal generator  $\LL_{i_0}+q_{i_0i_0}$.
			At the lifetime $\tau_1$ of the killed process $\wt X^{i_0}$, jump to
plane $j\ne i_0$
			with probability $-q_{i_0j} (X^{i_0}(\tau_1-))/q_{i_0i_0} (X^{i_0}(\tau_1-))$
			and
			run an independent copy of
			a subprocess $\wt X^j$ of $X^j$ with killing rate $-q_{jj}(x)$ from position $X^{i_0}(\tau_1-)$. Repeat this procedure. The resulting process $\(X_t, \Lambda_t\)$
			is a strong Markov process with
			  lifetime $\zeta$ by \cite{INW, M}.
For each $x\in  \R^d$,
we say that the matrix $Q(x)$  is {\it Markovian} if $\sum_{j\in \M}q_{ij}(x)=0$ a.e. on
			  $\R^d$ for every $i\in \M$,
				and
{\it sub-Markovian} if	$\sum_{j\in \M}q_{ij}(x)\leq 0$ a.e. on
			  $\R^d$ for every $i\in \M$.
 When $Q(x)$ is {\it Markovian}, the lifetime
 $\zeta =\infty$, and when
 $Q(x)$ is just
{\it sub-Markovian}, $\zeta$ can be finite.
			  We use the convention that $(X_t, \Lambda_t)=\partial$ for $t\geq \zeta$,
			  where $\partial $ is a cemetery point,
				and any function is extended to $\partial$ by taking value zero there.
			It is easy to check that the law of $\(X_t, \Lambda_t\)$
			solves the martingale problem for $\(\mathcal{G}, C^2_b(\R^d)\)$ so  it is
			the desired switched jump-diffusion. This way of constructing switched diffusion
			has been utilized in \cite[p.296]{CZ}.  It follows from \cite{Wang14} that   law of $\(X_t, \Lambda_t\)$
			is the unique solution to the martingale problem for $\(\mathcal{G}, C^2_b(\R^d)\)$.

(b)	  Conditions (A1) and (A2) presents
	the uniform ellipticity of $a(x, i)$ and the uniform boundedness of $b(x, i)$ and $q_{ij}(x)$.
	 The measure $\pi_i(x, dz)$
	can be thought of as the intensity of the number of jumps from $x$ to $x+z$ (see \cite{BK05, BL02}).
Condition
(A4) tells us that $\pi_i(x, dy)$ is absolutely continuous with respect to the Lebesgue measure $dx$ on $\R^d$, and
the intensities of jumps from $x$ and $y$ to a point $z$ are comparable if
$x$, $y$ are relatively far from $z$ but relatively close to each other. If $\wdt\pi_i(x, z)$  is such that
$$
\frac{c_i^{-1}}{|z|^{d+\alpha_i}} \leq \wdt\pi_i (x, z) \leq \frac{c_i }{|z|^{d+\alpha_i}}
$$
for some $c_i\geq 1$ and $\alpha_i\in (0, 2)$,  then  condition (A4) is
satisfied with $1< \al_r< \kappa_2$ independent of $r\in (0, 1)$.
Condition (A4) is an essential hypothesis in the proof of the Harnack inequality.
 }\end{remark}

Throughout the paper,  we use capital letters $C_1, C_2, \dots$ for constants appearing in the statements of the results, and lowercase letters $c_1, c_2, \dots$ for constants appearing in proofs. The numbering of the latter constants afresh in every new proof.

\section{Maximum Principle}\label{sec:maximum}

In this section, we
establish maximum principle for the coupled system under  conditions (A1)-(A3).
We emphasize that we do not assume condition (A4) for the maximum principle.
In Subsection \ref{S:3.1}, we prepare three propositions for general diffusions with jumps that will be used several times in the sequel.

\subsection{Jump Diffusions and Strict Positivity}\label{S:3.1}

Consider
\begin{eqnarray}
 \mathcal{L} f(x) &=&\sum^d_{k,l=1}a_{kl}(x ){\frac{\p^2f(x )}{\p x_k\p x_l}+\sum^d_{k=1}b_{k}(x )\frac{\p f(x )}{\p x_k}}
\nonumber\\
 && +\int_{\R^d}  \left( f(x+z )-f(x )-\nabla f(x )\cdot z \indi_{\{|z|<1\}}\right)
 \pi (x,dz),
\end{eqnarray}
where $(a_{kl}(x))$ is a continuous matrix-valued function that is uniformly
elliptic and bounded,
 $b(x)=(b_1(x), \dots, b_d (x))$ is a bounded
$\R^d$-valued function on $\R^d$, and $\pi (x, dz)$ is a $\sigma$-finite measure on $\R^d$  satisfying
$$
K:=\int_{\R^d} \left( 1\wedge |z|^2 \right) \sup_{x\in \R^d} \pi (x, dz) <\infty.
$$
By \cite[Theorem 5.2]{Komatsu}, there is a unique conservative strong Markov process  $X=\{X_t, t\geq 0; \P_x, x\in \R^d\}$
that is the unique solution to the martingale problem $(\LL, C^2_b(\R^d))$.  Suppose $q\geq 0$ is a bounded function on $\R^d$.
One can kill the sample path of $X$ with rate $q$. For this, let $\eta$ be an independent exponential random variable with mean $1$.
Let
$$
\zeta =\inf \left\{t>0: \int_0^t q(X_s) ds >\eta \right\}
$$
and define
$Z_t=X_t$ for $t<\zeta$ and $Z_t=\partial$ for $t\geq \zeta$, where $\partial$ is a cemetery point.
It is easy to see that for any   $x\in \R^d$ and $\varphi \geq 0 $ on $\R^d$,
$$
\E_x [ \varphi (Z_t); t<\zeta] = \E_x \left[ e_q(t) \varphi (X_t) \right], \quad t\geq 0,
$$
where
$$
e_q(t):= \exp \left( - \int_0^t q(X_s) ds \right).
$$
The process $Z$ is called   the subprocess of $X$ killed at rate $q$, and $\zeta$ the lifetime of $Z$.
For $A\subset \R^d$, we define its hitting time and exit time of $Z$ by
$$
\sigma^Z_A=\inf\{t\geq 0: Z_t\in A\} \quad \hbox{and} \quad
     \tau^Z_A=\inf\{t\geq 0: Z_t\notin A\} ,
 $$
with the convention that $\inf \emptyset =\infty$. Note that $\tau^Z_A\leq \zeta$.
The following two  propositions are based on the support theorem for  diffusions with jumps in \cite{Foondun09}.

\begin{proposition}\label{P:4.6.2}
Let $\lambda \geq 1$ be so that $\lambda^{-1} I_{d\times d} \leq (a_{kl}(x)) \leq \lambda I_{d\times d} $.
There is a   positive constant $C_1$ depending only on $\lambda$ and an upper bound on $\| b\|_\infty$
and $\| q\|_\infty$    such that for   any $R\in (0, 1]$,  $r\in (0, R/4)$, $x_0\in \R^d$,   $x \in B(x_0, 3R/2) $
and  $y\in B(x_0, 2R)$,
	$$\P_{y  }\(\sigma^Z_{B(x, r )}<\tau^Z_{B(x_0, 2R)}\)\ge C_1r^6.$$
\end{proposition}

\begin{proof} Note that  	 $Z_t=X_t$ for $t\in [0, \zeta)$, where $\zeta$ is
		the lifetime of $Z$.	Define
	$$
	\sigma_{B(x, r )} =\inf\{t\ge 0: X_t\in B(x, r ) \}, \quad \tau_{B(x_0, 2R)}=\inf\{t\ge 0: X_t\notin B(x_0, 2R) \}.
	$$
	Define a function $\phi:[0, 8]\mapsto \R^d$ as follows
	$$\phi(t)=y+\dfrac{x-y}{|x-y|}t, \quad t\in [0, 8].$$
	By \cite[Theorem 4.2]{Foondun09} and \cite[Remark 4.3]{Foondun09}, there exists a constant $c_1>0$
	 so  that
	\beq{E:17.4}
	\P_y\Big(\sup\limits_{t\le 8 }|X_t-\phi(t)|<r \Big)\ge c_1 r^6,
	\eeq
	for any $x\in B(x_0, 3R/2) $ and $r\in (0, R/4)$. Moreover, $c_1$ depends only on $\lambda$ and an upper bound
	on $\| b\|_\infty$  and $\| q\|_\infty$.
	Since   $|\phi'(t)|=1$ and $\phi (|x-y|)=x$,  on $\big\{\sup\limits_{t\leq  8}|X_t-\phi(t)|<r \big\}$, we have
	$X_{|x-y|} \in B(x, r )$,  $X_t \in B(x_0, 2R)$  for  $0\leq t\leq |x-y|$, and $|X_{8R}-x_0|\geq |X_{8R}-y|-|y-x_0|>3R$.
 	As a result,
	$\sigma_{B(x, r)} <|x-y|<\tau_{B(x_0, 2R)}< 8R\leq 8$ on $ \big\{\sup\limits_{t\leq  8}|X_t-\phi(t)|<r \big\}$.
	Then \eqref{E:17.4} leads to
	$$
	\P_y\Big( \sigma_{B(x, r)} <\tau_{B(x_0, 2R)} <8  \Big)\ge c_1r^6.
	$$
	Since  $\P_{y  }(\zeta>8) \ge   \exp(-6  \|q\|_\infty)$,
	 	we have
	\begin{eqnarray*}
	\P_{y }\(\sigma^Z_{B(x, r)}<\tau^Z_{B(x_0, 2R)}\) &\geq & \P_y\left( \sigma^Z_{B(x, r)} <\tau^Z_{B(x_0, 2R)} <8 <\zeta \right) \\
&\ge& \exp(-6  \|q\|_\infty) \P_y\Big( \sigma_{B(x, r)} <\tau_{B(x_0, 2R)} <8 \Big) \\
& \ge& \exp(-6  \|q\|_\infty)  c_1r^6.
	\end{eqnarray*}
This proves the proposition.
\end{proof}

\begin{proposition}\label{P:4.6}
{\rm (i)}  For any $0<r\leq 1/2$ and $x_0\in \R^d$, if $A\subset B(x_0, r)$ has positive Lebesgue measure, then
$\P_{x}(\sigma^Z_A<\tau^Z_{B(x_0, 2r)})>0$ for every $x\in B(x_0, r)$.

{\rm (ii)}    Let $\rho\in (0,1)$ be a constant.
There exist a nondecreasing  function $\Phi: (0, \infty)\mapsto (0, \infty)$ and $ r_0\in (0, 1/2]$ such that
for any $ x_0  \in \RR^d $, any  $r\in (0, r_0)$, and any Borel subset $A$ of
$B(x_0,r)$ with $|A|/r^d\ge \rho$, we have
\begin{equation}\label{E:17}
\P_{x}\Big(\sigma^Z_A<\tau^Z_{B(x_0, 2r)}\Big)\geq  \dfrac{1}{2}\Phi\(|A|/r^d\),\quad x\in B(x_0,r).
\end{equation}
\end{proposition}

\begin{proof} As in the proof of Proposition \ref{P:4.6.2},
 	define
	$$\
	\sigma_A=\inf\{t\ge 0: X_t \in A \} \quad \text{and} \quad  \tau_{B(x_0, 2r)} =\inf\{t\ge 0: X_t\notin B(x_0, 2r) \}.
	$$
By \cite[Corollary 4.9]{Foondun09}, there is a nondecreasing  function $\Phi: (0, \infty)\mapsto (0, \infty)$
	such that if $A\subset B(x_0, r)$, $|A|>0$, $r\in (0, 1/2]$
and $x\in B(x_0, r)$, then
\beq{E:17.2}
\P_{x}(\sigma_A<\tau_{B(x_0, 2r)})\ge \Phi\( {|A|}/ {r^d}\).
\eeq	
Using test function and  It\^{o}'s formula, it is easy to derive (see \cite[Proposition 3.4(b)]{Foondun09}
or Proposition \ref{L:4.5} below) that there is a constant $c_1>0$ independent of $x_0$ and $r\in (0, 1/2]$ so that
\begin{equation}\label{e:3.15}
\E_x \tau_{B(x_0, 2r )}\le c_1 r^2 \quad \hbox{for any } x\in B(x_0, 2r).
\end{equation}

(i) Suppose $0<r\leq 1/2$ and $A\subset B(x_0, r)$ has positive Lebesgue measure. Then by \eqref{E:17.2},
$\P_{x}(\sigma_A<\tau_{B(x_0, 2r)})>0$. Hence in view of \eqref{e:3.15}, we have for every $x\in B(x_0, r)$,
$$
\P_{x}(\sigma^Z_A<\tau^Z_{B(x_0, 2r)})
 \geq  \P_{x}(\sigma_A<\tau_{B(x_0, 2r)}<\zeta)  = \E_x \left[ e_q(\tau_{B(x_0, 2r)})  {\bf 1}_{\{  \sigma_A<\tau_{B(x_0, 2r)}\}}\right] >0.
$$

(ii) Observe that
\beq{E:17.1}\barray
	\P_x\Big(\sigma^Z_A<\tau^Z_{B(x_0, 2r)}\Big)\ad \ge \P_x\Big(\sigma_A<\tau_{B(x_0, 2r)}; \tau_{B(x_0, 2r)}<\zeta\Big)
	 \\
\ad \ge \P_{x}\Big(\sigma_A<\tau_{B(x_0, 2r)}\Big)-\P_x\Big( \tau_{B(x_0, 2r)}\ge \zeta\Big)
\earray\eeq
For $A\subset B(x_0,  r)$ with $|A|\ge \rho r^d$, we have  $\P_{x}(\sigma_A<\tau_{B(x_0, 2r)})\ge \Phi\(\rho\).$
On the other hand,
$$
\P_{x}(\zeta> t) = \E_x\left[ \exp \left(- \int_0^t q(X_s) ds\right)\right] \ge  \exp(-\| q\|_\infty \, t).
$$
This combined with \eqref{e:3.15} yields  that
\bea
\P_{x } (\zeta> \tau_{B(x_0, 2r)} )\ad \ge \P_{x }  (\zeta>r>\tau_{B(x_0, 2r)}  )\\
\ad \ge \P_{x }  ( \zeta>r ) - \P_{x  }  ( \tau_{B(x_0, 2r)}\ge r )\\
\ad \ge \exp(-\|q\|_\infty r) - \dfrac{\E_{x }\tau_{B(x_0, 2r)}}{r} \\
\ad \ge \exp(- \| q\|_\infty r)-c_1r.
\eea
Since $\lim_{r\to 0} \left( \exp(-\| q\|_\infty r)-c_1r\right)=1$,
  there is a constant  $ r_0\in (0, 1/2]$  such that
  \beq{E:17.3}
  \P_{x}\Big( \tau_{B(x_0, 2r)}\ge \zeta\Big)\le \dfrac{1}{2}\Phi(\rho) \quad \text{for all }    r\in (0, r_0).
  \eeq
The desired conclusion follows from \eqref{E:17.1},
\eqref{E:17.2}, and \eqref{E:17.3}.
\end{proof}

\medskip

For a connected open subset $D\subset \R^d$ and a Borel measurable function $f\geq 0$ on $D$, define $G^q_D f(x)
=\E_x \left[\int_0^{\tau^Z_D} f(Z_s) ds\right]$.

\begin{proposition}\label{P:3.7} For $f\geq 0$, either $G^q_D f(x)>0$  on $D$ or $G^q_D f(x)\equiv 0$   on $D$.
Moreover, if $G^q_Df>0$ on $D$ if and only if $\{x\in D: f(x)>0\}$ has positive Lebesgue measure.
\end{proposition}

\begin{proof} Suppose that $A:=\{x\in D: G^q_D f(x)>0\}$ has positive Lebesgue measure. We claim that
	for any $r\in (0, 1]$ and $B(x_0, r)\subset D$ so that $B(x_0, r/2)\cap A$ has positive Lebesgue measure, then $B(x_0, r/2)\subset A$.
	This is because if $B(x_0, r/2)\cap A$ has positive Lebesgue measure, then   there is a compact subset $K\subset B(x_0, r/2) \cap A$
	having positive Lebesgue measure. By  Proposition  \ref{P:4.6} (i), we have
$\P_x (\sigma^Z_K<\tau^Z_{B(x_0, r)} )>0$ for every $x\in B(x_0, r/2)$.
Consequently,
$$
G_D^q  f(x)= \E_x\int_0^{\tau^Z_D} f(Z_s)ds \geq \E_x \left[ G^q_D f(Z_{\sigma_K}); \sigma^Z_K<\tau^Z_{B(x_0, r)}\right]
>0
$$
for every $x\in B(x_0, r/2)$.  This proves the claim. Since $B(x_0, r/2)\subset A$, by a chaining argument, the above reasoning shows that $A=D$
if $A$ has positive Lebesgue measure.  Now assume that $G^q_D f=0$ a.e. on $D$.
Since $G^q_D f(x)=\E_x \int_0^{\tau_D} e_q(s) f(X_s)ds$, we have
$G_D f(x):= \E_x \int_0^{\tau_D}   f(X_s)ds=0$ a.e. on $D$.
In particular, $G_D (f\wedge n)=0$ a.e. on $D$.
By \cite[Theorem 2.3]{Foondun09},  bounded harmonic functions of $X$ is H\"older continuous.
By the proof of \cite[Proposition 3.3]{BKK}, this together with \eqref{e:3.15} implies that
$G_D (f\wedge n)$ is H\"older continuous on $D$. Therefore we have $G_D (f\wedge n) (x)=0$
for every $x\in D$. Consequently, $G_Df(x)=0$ for every $x\in D$ and so is $G^q_Df(x)$.
This proves the first part of the proposition.

For the second part of the proposition, suppose that $f\geq 0$ and  $f=0$ a.e. on $D$. It follows from \cite[Corollary 2]{MP} that for every $x_0\in \R^d$,
\begin{equation}\label{e:3.18}
\E_x \int_0^{\tau_{B(x_0, 1)} } ({\bf 1}_D f ) (X_s) ds =0 \quad \hbox{for every } x\in B(x_0, 1).
\end{equation}
We claim that $\E_x \int_0^{\tau_D} f(X_s) ds=0$ for every $x\in D$.
For this, we define a sequence of stopping times: $\tau_0:=0$, $\tau_1:=\inf\{t\geq 0: |X_t-X_0|\geq 1\}\wedge \tau_D$,
and for $n\geq 2$, $\tau_n :=\inf\{t\geq \tau_{n-1}: |X_t-X_{\tau_{n-1}}|\geq 1\}\wedge \tau_D$.
Note that on $\{\lim_{n\to \infty} \tau_n <\tau_D\}$,
$\lim_{n\to \infty} X_{\tau_n}=X_{\lim_{n\to \infty}} $ by the left-continuity of $X_t$.
On the other hand, the sequence $\{X_{\tau_n}; n\geq 1\}$ diverges on
$\{\lim_{n\to \infty} \tau_n <\tau_D\}$
as $|X_{\tau_n}-X_{\tau_{n-1}}|\geq 1$  by the definition of $\tau_n$.
This contradiction implies that $\P_x  (\lim_{n\to \infty} \tau_n <\tau_D)=0$;
 in other words, $\lim_{n\to \infty} \tau_n =\tau_D$   $\P_x$-a.s.
Consequently, we have by \eqref{e:3.18}
\begin{eqnarray*}
 \E_x \int_0^{\tau_D} f(X_s) ds &=& \E_x \sum_{n=1}^\infty \int_{\tau_{n-1}}^{\tau_n} f(X_s) ds  \\
&=&\sum_{n=1}^\infty \E_x \left[ \E_{X_{\tau_{n-1}}} \int_0^{\tau_{B(X_{\tau_{n-1}}, 1)}\wedge \tau_D} f(X_s)ds; \tau_{n-1}<\tau_D \right]
\\
&=& 0.
\end{eqnarray*}
It follows then $G^q_D f(x) = \E_x \int_0^{\tau_D} e_q(X_s) f(X_s) ds =0$ for every $x\in D$.
This proves that if $f\geq 0$ and $f=0$ a.e. on $D$, then $G^q_D f\equiv 0$ on $D$.
Next suppose that $f\geq 0$ is a bounded function on $\R^d$ and $\{x\in D: f(x)>0\}$ has positive Lebesgue measure,
we will show that $G^q_D(x)>0$ for every $x\in D$. Let $c_p>0$ be the constant in the Remark  following
Theorem 3.1 on p.282 of \cite{Komatsu}.
Using a localization argument if needed, we   may
assume that $|a_{ij}(x)-a_{ij}(y)| \leq 1/c_p$ for every $x, y\in \R^d$.
Let $K$ be a compact subset of $D$ so that $\{x\in K: f(x)>0\}$ has positive Lebesgue measure.
Then by Theorem 3.6 and the proof of Theorem 4.2 both in \cite{Komatsu},  for $\lambda>0$ large,
$v(x):=\E_x \int_0^\infty e^{-\lambda t} ({\bf 1}_K f)(X_s) ds $ is non-trivial on $\R^d$.
We define a sequence of stopping times as follows.
Let $S_1:=\sigma_K$, $T_1:=\inf\{t>\sigma_K: X_t\notin D\}$; for $n\geq 2$,
define $S_n:=\inf\{t>T_{n-1}: X_t\in K\}$ and $T_n:=\inf\{t>S_n: X_t\notin D\}$.
Then
$$
v(x)= \sum_{n=1}^\infty \E_x \int_{S_n}^{T_n} e^{-\lambda s} f(X_s) ds
= \sum_{n=1}^\infty \E_x \left[ e^{-\lambda S_n} G_{D, \lambda} ({\bf 1}_K f) (X_{S_n}) \right],
$$
where $G_{D, \lambda} \varphi (x):= \E_x \int_0^{\tau_D} e^{-\lambda s} \varphi (X_s) ds $.
Hence $G_{D, \lambda} ({\bf 1}_K f) (x) $ cannot be identically zero on $K$.
By the first part of this proof (by taking $q=\lambda$), we have $G_{D, \lambda} ({\bf 1}_K f) (x) >0$
for every $x\in D$. It follows that $G_Df(x)>0$ and so $G^q_Df (x)>0$ for every $x\in D$.
\end{proof}

\subsection{Maximum Principle for Switched Markov Processes}

Now we return to the setting of switched Markov process  $(X_t, \Lambda_t)$.
Let $D$ be a bounded open set in $\RR^d$ and $U=D\times \M$.
 Then $\tau_D=\inf\{t>0: X_t \notin D\}$ is the same as
$\tau_U:=\inf\{t>0: Y_t:=(X_t, \Lambda_t) \notin U\}$.
 Suppose $u$ is a $\mathcal{G}$-harmonic function in $U$. Under some mild assumptions
 (for example, when $u$ is bounded and continuous up to $\partial D\times \M$), we have
 \begin{equation}\label{E:10}
 u(x, i)=\E_{x, i} [ u\(X_{\tau_D}, \Lambda_{\tau_D}\)] \qquad \hbox{for } (x, i)\in U.
 \end{equation}
It follows immediately that if $u\geq 0$ on $U^c$, then $u\geq 0$ in $U$.

To proceed, we recall the notion of  irreducibility  of the generator $\G$ or the matrix function $Q\cd$.
The operator $\mathcal G$ or the matrix function $Q\cd$ is said to be \textit{irreducible on} $D$ if for any $i, j\in\M$, there exist $n=n(i,j)\geq1$ and $\Lambda_0, \dots, \Lambda_n\in\M$ with $\Lambda_{k-1}\neq\Lambda_k$ for $1\leq k\leq n$, $\Lambda_0=i, \Lambda_n=j$ such that
$\{x\in D: q_{\Lambda_{k-1}\Lambda_k}(x)>0 \}$
has positive Lebesgue measure
for $k=1, \dots, n$.

For each $i\in \M$, denote by $X^i$ the jump diffusion that solves the martingale problem  $(\LL_i, C^2_b(\R^d))$ and
$\wt X^i$ the subprocess of $X^i$ killed at rate $-q_{ii}(x)$.
 For a connected open set $D\subset \R^d$,  $  G^i_D$ denotes the Green operator of  $\wt X^i$ in $D$.

\begin{theorem}\label{T:3.8} Assume that conditions {\rm(A1)-(A3)} hold, that  $D$ is a bounded connected open set in $\R^d$,
and that $Q$ is irreducible on $D$.
 Suppose that $u$ is a $\mathcal{G}$-harmonic function in $U=D\times \M$ given by
$$
 u(x, i)=\E_{x, i} [ \phi \(X_{\tau_D}, \Lambda_{\tau_D}\); \tau_D<\infty] \qquad \hbox{for } (x, i)\in U
$$
 and $\phi \geq 0$ on $D^c\times \M$.
 Then either
 $u(x, i)>0$ for every $(x, i)\in U$ or $u\equiv 0$   on $U$.
\end{theorem}

\begin{proof}  Clearly $u\geq 0$ on $U$. Suppose that $u$ is not  a.e. zero on $U$. Without loss of generality,
let us assume that $\{x\in D: u(x, 1)>0\}$ has positive Lebesgue measure.
Denote by
$\tau_1:=\inf\{t>0: \Lambda_t \not= \Lambda_0\}$
  the first switching time for $Y_t=(X_t, \Lambda_t)$.
Let
$$v_i(x):= v (x, i): = \E_{x, i} [ \phi (  \wt X^i_{\tau_D}, i) ] = \E_{x, i} [ \phi (  X^i_{\tau_D}, i); \tau_D <\tau_1] .
$$
 Then
$v_i$ is a harmonic function of $\LL_i+q_{ii}$ in $D$ with $v_i= \phi (\cdot, i)$
on $D^c$.
For $1\leq i\leq m$, using the strong Markov property $\tau_1$, we have
\begin{equation}\label{E:12}
u(x, i) = v_i (x) + \sum_{j=1\atop{j\not=i}}^m   G_D^i (q_{ij} u(\cdot, j)) (x).
\end{equation}
 Under the above assumption,
  either $ \{x\in D: v_1 (x)>0\} $ or $\{x\in D:   G^1_D (\sum_{j=1\atop{j\not=i}}^m  q_{ij} u(\cdot, j)) (x)>0\}$
  has positive Lebesgue measure.
If the latter happens, then by Proposition \ref{P:3.7},  \hfill \break
$G^1_D (\sum_{j=1\atop{j\not=i}}^m  q_{ij} u(\cdot, j))  (x) >0$ and
hence $u(x, 1)>0$ for every $x\in D$.
Note that
\begin{equation}\label{e:3.22}
v_i (x)
= \E_{x} [ e_{-q_{ii}}(\tau_D) \phi (  X^i ({\tau_D}), i)] \leq \E_{x} [   \phi (  X^i ({\tau_D}), i)] =:
   \wt u_i(x) .
  \end{equation}
 Suppose $|\{x\in D: v_1 (x)>0\}|>0$. Then so does $A:=\{x\in D: \wt u_1 (x)>0\}$.
For any  $x_0\in D$ and $r\in (0, 1)$ so that $B(x_0, r)\subset D$ and $B(x_0, r/2)\cap A$
has positive Lebesgue measure, let $K\subset B(x_0, r/2)\cap A$ be a compact set having positive Lebesgue measure.
By \eqref{E:17.2}, $\P_x (\sigma^1_K <\tau^1_{B(x_0, r))}>0$ for every $x\in B(x_0, r/2)$, where
$\sigma^1_K :=\inf\{t\geq 0: X^1_t \in K\}$ and $\tau^1_{B(x_0, r)}:=\inf\{t\geq 0: X^1_t\notin B(x_0, r)\}$.
Hence for every $x\in B(x_0, r/2)$, by the strong Markov property of $X^i$ at $\sigma^1_K$,
$$ \wt u_1(x)\geq \E_x [ \wt u_1 (X^i_{\sigma^1_K}); \sigma^1_K <\tau_D] >0.
$$
Consequently, $B(x_0, r/2)\subset A$.
By the chaining argument, the same reasoning as above leads to $A=D$; that is, $\wt u_1(x)>0$ on $D$.
 By the probabilistic representation \eqref{e:3.22} of $v_1$, we have $v_1(x)>0$ on $D$ and hence $u(x, 1)>0$ on $D$.
Thus we have shown that  $u(x, 1)>0$ on $D$ whenever
$\{x\in D: u (x, 1)>0\}$ has positive Lebesgue measure.

For $i\not = 1$, there is a self-avoiding path $i=i_0, \dots,
i_{n}=i$ so that
$\{x\in D: q_{i_{k-1}i_k}(x)>0\}$
 has
 positive Lebesgue measure for each $k=1, \dots, n$.
 By \eqref{E:12} and its iteration, we have
\begin{eqnarray*}
u(x, i) &=& v_i (x) + \sum_{k=1}^n \sum_{l_1, \dots, l_k=1\atop{l_1\not=1, \l_2\not= l_1, \dots, l_k\not=l_{k-1}}}^m
 G_D^i(q_{il_1} (G_D^{l_1} q_{l_1l_2}(\cdots  ( G_D^{l_{k-1}}q_{l_{k-1}l_k} v_{l_k})) \cdots )))  (x) \\
&& + \sum_{l_1, \dots, l_n=1\atop {l_1\not=1, \l_2\not= l_1, \dots, l_n\not=l_{n-1}}}^m
 G_D^i (q_{il_1} ( G_D^{l_1}q_{l_1l_2} ( \cdots (G_D^{l_{n-1}} q_{l_{n-1}l_n} u(\cdot, l_n))
 \cdots ))) (x)\\
 &\geq &  G_D^i (q_{ii_1} ( G_D^{i_1}q_{i_1i_2} ( \cdots (G_D^{i_{n-1}} q_{i_{n-1}i_n} u(\cdot, i) ) \cdots ))) (x),
\end{eqnarray*}
which is strictly positive in $D$ by Proposition \ref{P:3.7}.

Next assume that $u=0$ a.e. on $U$. We claim that $u\equiv 0$ on $U$. In view of
\eqref{E:12} and Proposition \ref{P:3.7}, it suffices to show that $v_i(x)\equiv 0$ on $D$
for every $i\in \M$. Since
$$
v_i(x)=  \E_{x, i} \left[ \phi (  X^i_{\tau_D}, i); \tau_D <\tau_1 \right]
 = \E_{x, i} \left[ e^{(i)}_{-q_{ii}}(\tau_D)\phi (  X^i_{\tau_D}, i) \right],
$$
and $v_i(x)=0$ a.e. on $D$, where $$e^{(i)}_{-q_{ii}} (t):=\exp( \int_0^t q_{ii}(X^i_s)ds),$$
we have
$u_i(x):= \E_{x, i}\left[  \phi (  X^i_{\tau_D}, i) \right]$ vanishes a.e. on $D$.
The function $u_i(x)$ is harmonic in $D$ with respect to $X^i$ (or equivalently, with respect to the operator
$\LL_i$ in $D$).   By \cite[Theorem 2.3]{Foondun09},
 it is H\"older continuous in $D$. Hence $u_i(x)=0$ for every $x\in D$,
and so is $v_i(x)$. This proves that $u(x, i)=0$ for every $x\in D$ and every $i\in \M$.
\end{proof}

 \begin{theorem}\label{T:3.9} {\rm (Strong Maximum Principle I)}
 Assume conditions {\rm(A1)-(A3)} hold, $D$ is a bounded
 connected open set  in $\R^d$, and $Q(x)$ is irreducible on $D$. Suppose
 $u$ is a $\mathcal{G}$-harmonic function in $U=D\times \M$ given by
$$
 u(x, i)=\E_{x, i} [ \phi \(X_{\tau_D}, \Lambda_{\tau_D}\)] \qquad \hbox{for } (x, i)\in U
$$
for some $\phi $ with $M:= \sup_{(y, j)\in D^c \times \M} \phi (y, j)\in [0, \infty)$.
 If   $(x_0, i_0)\in D\times \M$ and $u(x_0, i_0) =M$,  then $u\equiv M$ on $D\times \M$.
If in addition $M>0$, then the matrix $Q(x)$ is Markovian.
\end{theorem}

\begin{proof}
(i) First we assume that $Q(x)$ is Markovian in the sense that $\sum_{j\in \M} q_{ij}(x)=0$
a.e. on $\R^d$ for every $i\in \M$. In this case, by the construction of the switched Markov process
$\( X_t, \Lambda_t\)$ outlined in Remark \ref{R:2.2}(a), $\( X_t, \Lambda_t\)$ has infinite lifetime and so
constant 1 is a $\G$-harmonic function on $\R^d$. Hence
$$
M-u(x, i)= \E_{x, i} \left[ (M-\phi)  (X_{\tau_D}, \Lambda_{\tau_D}); \tau_D<\infty \right]
 $$
 is a non-negative $\mathcal{G}$-harmonic function in $D\times \M$.
 (Note that $\tau_D<\infty$ $\P_{x,i}$-a.s. in view of Proposition \ref{L:4.5} below.)
Since $ u(x_0, i_0)=M$, we have by Theorem \ref{T:3.8} that  $u(x, i)=M$ for every
$(x, i)\in U$.

(ii) We now consider the general case that $Q(x)$ is a sub-Markovian matrix.
Define a Markovian matrix $\bar Q(x)=(\bar q_{ij}(x))$
by taking
$\bar q_{ij}(x)=q_{ij}(x)$ and $\bar q_{ii}(x)=-\sum_{j\in \M\setminus \{i\}} q_{ij}(x)$.
 Let $(\bar X_t, \bar \Lambda_t)$ be the conservative switched Markov process
corresponding to $\bar \G$ as in \eqref{E:1} but with $\bar Q(x)$ in place of $Q(x)$.
The original switched Markov process $(X_t, \Lambda_t)$ can be viewed as a subprocess
of $ \(\bar X_t, \bar \Lambda_t\)$ killed at rate $\kappa (x, i)=\bar q_{ii}(x)-q_{ii}(x)$;
that is, for every $\psi (x, i)\geq 0 $ on $\R^d\times \M$,
$$
\E_{x, i} [ \psi (X_t, \Lambda_t)] =\E_{x, i} \left[ \bar e_\kappa (t) \psi (\bar X_t, \bar \Lambda_t) \right],
$$
where $\bar e_\kappa (t)=\exp  (-\int_0^t \kappa (\bar X_s, \bar \Lambda_s) ds  )$.
We consider
\begin{equation}\label{e:3.12a}
v(x, i):=\E_{x, i} \left[ (M-\phi)  (X_{\tau_D}, \Lambda_{\tau_D}); \tau_D<\infty \right],
\end{equation}
which is a non-negative $\mathcal{G}$-harmonic function in $D\times \M$.
We can rewrite $v(x, i)$ on $U$ as
\begin{equation}\label{e:3.13a}
v(x, i)=
 \E_{x, i} [ \bar e_{\kappa} (\tau_D) (M- \phi) (\bar X_{\tau_D}, \bar \Lambda_{\tau_D} )].
\end{equation}
 Since $u(x_0, i_0)=M\geq 0$,
 $$
 0\leq v(x_0, i_0)
 = M \P_{x_0, i_0}(\tau_D <\infty)- u(x_0, i_0)\leq 0,
 $$
  that is, $v(x_0, i_0)=0$.
 Thus by Theorem \ref{T:3.8},  $v(x, i)\equiv 0 $ on $D\times \M$.
 This implies by \eqref{e:3.13a} that $  (M- \phi) (\bar X_{\tau_D}, \bar \Lambda_{\tau_D} )=0$
 $\P_{x, i}$-a.s. for every $(x, i)\in D\times \M$. Consequently, we have
 \begin{equation}\label{e:3.14a}
  u(x, i)= M \E_{x, i} \left[ \bar e_\kappa (\tau_D) \right] \quad \hbox{for } (x, i)\in D\times \M
 \end{equation}
   and so
 \begin{eqnarray*}
 u(x, i)&=& M + M \E_{x,i}[  \bar e_{\kappa} (\tau_D) -1)  ] \\
 &=& M - \E_{x,i} \left[  \int_0^{\tau_D} \kappa ( \bar X_s, \bar \Lambda_s) \exp\left(
 -\int_s^{\tau_D} \kappa  (\bar X_r, \bar \Lambda_r) dr \right)   ds \right] \\
  &=& M-  M \E_{x,i} \left[  \int_0^{\tau_D}  \kappa   ( \bar X_s, \bar \Lambda_s)  ds \right].
   \end{eqnarray*}
 Let $\tau_1:=\inf\{t>0: \bar \Lambda_t\not= \bar \Lambda_0\}$ and denote by $\bar G^i_D$ be the Green function
 of $\LL_i+\bar q_{ii}$ in $D$.
 Since $u(x_0, i_0)=M$, we have by the strong Markov property and the construction of $(\bar X_t, \bar \Lambda_t)$ in
 Remark \ref{R:2.2}(a),
   \begin{eqnarray}
0 &=&   \E_{x_0,i_0} \left[  \int_0^{\tau_D}  \kappa  ( \bar X_s, \bar \Lambda_s)  ds\right]
\nonumber \\
&\geq&  \E_{x_0,i_0} \left[  \int_0^{\tau_D\wedge \tau_1} \kappa  ( \bar X_s, i_0)  ds \right]
=  \bar G^{i_0}_D (\kappa (\cdot, i_0))(x_0) \geq 0.  \label{e:3.12}
 \end{eqnarray}
Thus   $\bar G^{i_0}_D (\kappa (\cdot, i_0))(x_0) =0$ and so by  Proposition \ref{P:3.7} we have
 $\kappa (x, i_0)=0$ a.e. on $D$.
 Observe that
 \begin{eqnarray*}
 u(x_0, i_0) &=&\E_{x_0, i_0} [ \bar e_{\kappa} (\tau_D) \phi (\bar X_{\tau_D}, \bar \Lambda_{\tau_D})] \\
 &=& \E_{x_0, i_0} [ \phi (\bar X_{\tau_D}, \bar \Lambda_{\tau_D} ); \tau_D < \tau_1]
     +  \E_{x_0, i_0} [ \bar e_{\kappa} (\tau_D)  \phi (\bar X_{\tau_D}, \bar \Lambda_{\tau_D} ); \tau_1\leq  \tau_D  ].
\end{eqnarray*}
 Hence using the strong Markov property, we have
  \begin{eqnarray*}
 0&=&(M -u) (x_0, i_0)
 =  \E_{x_0, i_0} [ (M-M\bar e_{\kappa} (\tau_D)) ]\\
 &=&    \E_{x_0, i_0} [ (M-M\bar e_{\kappa} (\tau_D))  ; \tau_1\leq  \tau_D  ]\\
 &=& \E_{x_0, i_0} \left[  (M-u) (\bar X_{\tau_1}, \bar \Lambda_{\tau_1} ); \tau_1\leq  \tau_D  \right] \\
 &=& \sum_{j\in \M\setminus \{i_0\}}  \E_{x_0,i_0} \left[
 (M-u) (\bar X_{\tau_1-},   j))
  (q_{i_0j}/\bar q_{i_0i_0})  (\bar X_{\tau_1-}) ; \tau_1\leq  \tau_D \right]  \\
&  =&  \sum_{j\in \M\setminus \{i_0\}}  \bar G^{i_0}_D ( q_{i_0j} (M -u)(x_0).
\end{eqnarray*}
 By Proposition \ref{P:3.7} again, we have $\sum_{j\in \M\setminus \{i_0\}}  q_{i_0j} (M -u)=0$ a.e. on $D$.
 Since   $\bar Q$ is irreducible on $D$,  for any $j\not=i_0$, there is a self-avoiding path $\{j_0=i_0, j_1, \dots, j_n=j\}$
 so that   $\{x\in D: q_{j_kj_{k+1}} (x)>0\}$ having positive Lebesgue measure for $k=0. 1, \dots, n-1$.
 Thus we have $u(x, j_1)=M$ on $\{x\in D: q_{i_0j_1} (x)>0\}$.
 By the argument above, this implies that $\kappa (x, j_1)=0$ a.e. on $D$.
 Continuing as this, we get $\kappa (x, j_k)=0$ a.e. on $D$ and $\{x\in D: u(x, j_k)=M\}$ has positive Lebesgue measure
 for $k= 2, \dots, n$. This proves that $\kappa (x, i)=0$ a.e. on $D$ for every $i\in \M$ and so $u(x, i)=M$
 for every $x\in D$ in view of \eqref{e:3.14a}.
 \end{proof}

Before presenting the next version of strong maximum principle, we first prepare a lemma.

\begin{lemma}\label{L:3.6}
 Assume conditions {\rm(A1)-(A3)} hold,  $D$ is a bounded connected open set in $\R^d$,
and $Q(x)$ is irreducible on $D$. For any $\phi \geq 0 $ on $D\times \M$, either
$\E_{x,i} \int_0^{\tau_{D}} \phi (X_s, \Lambda_s) ds >0$ for every $(x, i)\in D\times \M$
or $\E_{x,i} \int_0^{\tau_{D}} \phi (X_s, \Lambda_s) ds \equiv 0$ on $D\times \M$.
\end{lemma}

\begin{proof} Denote by $ G^i_D$   the Green function of $\LL_i+\bar q_{ii}$ in $D$.
Using the strong Markov property at the first switching time
$\tau_1:=\inf\{t\geq 0: \Lambda_t \not= \Lambda_{t-}\}$ in a similar way to that for \eqref{e:3.12},
we have for every   $(x, i)\in D\times \M$,
\begin{eqnarray}
v(x, i)&:=&  \E_{x,i} \int_0^{\tau_{D}} \phi (X_s, \Lambda_s) ds \nonumber \\
&=&\E_{x,i} \int_0^{\tau_{D}\wedge \tau_1} \phi (X_s, \Lambda_s) ds
+ \E_{x,i} \left[ \int_{\tau_1}^{\tau_{D}\wedge \tau_1} \phi (X_s, \Lambda_s) ds; \tau_1 <\tau_D \right]
   \nonumber \\
&=& G^{i}_{D} (\phi (\cdot, i))(x) + \E_{x,i} \left[ v(X_{\tau_1}, \Lambda_{\tau_1}); \tau_1<\tau_D \right]
 \nonumber \\
&=& G^{i}_{D} (\phi (\cdot, i))(x) + \sum_{k\in \M \setminus \{i\}}
\E_{x,i} \left[ v(X_{\tau_1 -}, k) (q_{ik}/q_{ii})(X_{\tau_1-}); \tau_1<\tau_D \right]
\nonumber \\
&=&      G^{i}_{D} (\phi (\cdot, i))(x)  + \sum_{k\in \M\setminus \{i\}}     G^{i}_{D}
\left( q_{ik}   v (\cdot,   k) \right) (x), \label{e:3.13}
\end{eqnarray}
where the last identity is due to \cite[p.286]{Sharpe}; see the proof of \cite[Proposition 2.2]{CZ}.

Suppose $v(x_0, i_0)=0$ for some $(x_0, i_0)\in D\times \M$.
 Then by  Proposition \ref{P:3.7}, $v(x, i_0) \equiv 0$ on $D$.
For any $j\in \M\setminus \{i_0\}$,
since $Q(x)$ is irreducible on $D$,   there is a self-avoiding path $\{j_0=i_0, j_1, \dots, j_n=j\}$
 so that   $\{x\in D: q_{j_kj_{k+1}} (x)>0\}$ having positive Lebesgue measure for $k=0, 1, \dots, n-1$.
 It follows from \eqref{e:3.13} and its iteration that
 $$
 0=v(x_0, i_0)\geq   G^{i_0}_D (q_{i_0j_1} (  G^{j_1}_D q_{j_1j_2} ( \cdots ( G^{j_{n-1}}_D ( q_{j_{n-1}j}
 v(\cdot, j) ) \cdots ))) (x_0)\geq 0.
 $$
We conclude from Proposition \ref{P:3.7} that $q_{j_{n-1}j} (\cdot) v(\cdot , j)=0$ a.e. on $D$.
So there is some $y\in D$ so that $v(y, j)=0$. By \eqref{e:3.13} with $(y, j)$ in place of $(x, i)$
and Proposition \ref{P:3.7}, we have $v(x,j)=0$ for every $x\in D$.
\end{proof}

\begin{theorem}\label{T:strong}
{\rm (Strong Maximum Principle II)}
 Suppose that conditions {\rm(A1)-(A3)} hold, $D$ is a bounded
 connected open set  in $\R^d$ and $Q(x)$ is irreducible on $D$.
 If  $f(\cdot,i) \in C^2(D)$,
  $\sup_{\R^d\times \M} f \geq 0 $,
  and
$$
\mathcal G f(x,i)\ge 0\quad \hbox{for } (x, i)\in D\times \M,
$$
then $f(x,i)$ cannot attain its maximum inside $D\times\M$ unless
$$f(x,i)\equiv \sup\limits_{\R^d \times\M}f(y,j) \quad \hbox{on }  D\times\M.$$
\end{theorem}

\begin{proof} Suppose $f$ achieves its maximum at some $(x_0, i_0)\in D\times \M$.
Let $D_1$ be any relatively compact connected open subset of $D$ that contains $x_0$
and that $Q(x)$ is irreducible on $D_1$.
Then by It\^{o}'s formula, we have for every $(x, j)\in D_1 \times \M$,
\begin{eqnarray}
f(x, j) &=& \E_{x,j} \left[ f( X_{\tau_{D_1}}, \Lambda_{\tau_{D_1}})\right] - \E_{x,j} \int_0^{\tau_{D_1}} \G f (X_s, \Lambda_s) ds \nonumber \\
 &\leq&   \E_{x,j} \left[ f( X_{\tau_{D_1}}, \Lambda_{\tau_{D_1}}\right] =: h(x, j). \label{e:3.14}
\end{eqnarray}
Let $M= \sup_{(y, j)\in \R^d \times \M} f(y, j)$,
which is non-negative.
 In view of \eqref{e:3.14},
$$
M=\sup_{(y, j)\in D_1^c  \times \M} f(y, j) = f(x_0, i_0).
$$
Clearly, $h\leq M$  and $h$ is $\G$-harmonic in $D_1\times \M$.
We have by \eqref{e:3.13}, $h(x_0, i_0)=M$ and
$$
\E_{x_0,i_0} \int_0^{\tau_{D_1}} \G f (X_s, \Lambda_s) ds=0.
$$
Theorem \ref{T:3.9} and Lemma \ref{L:3.6} tell us that $h\equiv M$ on $D_1\times \M$
and $\E_{x,i} \int_0^{\tau_{D_1}} \G f (X_s, \Lambda_s) ds=0$ for every $(x, i)\in D_1\times \M$.
Consequently, $f(x, i)\equiv M$ on $D_1\times \M$. Letting $D_1$ increase to $D$  establishes the theorem.
\end{proof}

\section{Harnack Inequality}\label{sec:harnack}

This section is devoted to the Harnack inequality for $\mathcal{G}$-harmonic functions.
For simplicity, we introduce some notation as follows. For any $U
=D\times \M\subset \R^d\times \M$, recall that
$$ \tau_D=\inf\{t\ge 0: X_t\notin D\}.
$$
We define
$$T^i_D:=\inf\{t\ge 0: X_t\in D,  \Lambda_t=
i\}, \quad i\in \M.$$

\begin{proposition}\label{P:4.2}
Assume conditions {\rm(A1)-(A3)} hold.
There exists a constant $C_2$
not depending on $x_0\in \R^d$ such that for any $r\in (0, 1)$ and any $i\in \M$,
\beq{E:13}
\P_{x_0, i}\(\tau_{B(x_0, r)}\le C_2r^2\)\le 1/2.
\eeq
\end{proposition}

\begin{proof}	
Let $v(\cdot,i) \in C^2(\R^d)$ be a nonnegative function independent of $i$ and
\begin{eqnarray*}
v(x,i)=
\left\{
\begin{array}{ll}
|x-x_0|^2,\quad |x-x_0|\leq r/2,\\
r^2,\quad  |x-x_0|\geq r
\end{array}
\right.
\end{eqnarray*}
 such  that $v$ is bounded by $c_1r^2$, and its first and second order derivatives  are bounded by $c_1r$ and $c_1$, respectively.
  Since $\P_{x_0, i}$ solves the martingale problem, we have
$$
  \E_{x_0,i}v(X_{{t\wedge \tau_{B(x_0,r)}} },\Lambda_{{t\wedge \tau_{B(x_0,r)}} })
   =v(x_0,i)+\E_{x_0,i}\int_0^{t\wedge \tau_{B(x_0,r)}} \mathcal Gv(X_s,\Lambda_s)ds.
 $$
  Using the
 boundedness of the first and second derivatives of $v(\cdot, i)$ and $Q\cd$, we have
 $$\int_0^{t\wedge \tau_{B(x_0,r)}} \mathcal Gv(X_s,\Lambda_s)ds\le c_2t.$$
It follows that
$$\E_{x_0,i}v(X_{{t\wedge \tau_{B(x_0,r)}} },\Lambda_{{t\wedge \tau_{B(x_0,r)}} })-v(x_0,i)\leq c_2t.$$
On the other hand, since $v(X_{{\tau_{B(x_0,r)}} },\Lambda_{{ \tau_{B(x_0,r)}} })= r^2$, we obtain
\begin{equation*}
\E_{x_0,i}v(X_{{t\wedge \tau_{B(x_0,r)}} },\Lambda_{{t\wedge \tau_{B(x_0,r)}} }) \ge  r^2\P_{x_0,i}(\tau_{B(x_0,r)}\leq t).
\end{equation*}
Hence
 $$r^2\P_{x_0,i}(\tau_{B(x_0,r)}\leq t)\leq c_2t.$$
Taking $C_2=\frac{1}{2c_2}$ in the above formula and replacing $t$ by $C_2r^2$, we obtain \eqref{E:13}.
\end{proof}

\begin{proposition}\label{L:4.3}
Assume conditions {\rm(A1)-(A3)} hold.
 For any constant $\e\in (0,1)$, there exist positive constants $C_3$ and $C_4$ depending only on $\varepsilon$ such that for any $(x_0, i)\in \RR^d\times \M$ and any $r\in (0,1)$, we have
\begin{itemize}
	\item[\rm (a)]
	$\P_{x,i}\(\tau_{B(x_0,r)}> C_3 r^2\) \ge 1/2$ \ for $(x, i)\in B(x_0,(1-\varepsilon)r)\times \M$.
	\item[\rm (b)]
	$\E_{x,i}\tau_{B(x_0,r)}\geq C_4 r^2$ for
	$(x, i)\in B(x_0,(1-\varepsilon)r)\times\M$.
\end{itemize}
\end{proposition}

\begin{proof}
By Proposition \ref{P:4.2},
there exists a constant $c_1$ depending only on $\varepsilon$ such that for any $(x,i)\in B(x_0,(1-\varepsilon)r)\times\M$, we have
$$
\P_{x,i}(\tau_{B(x_0,r)}\leq c_1r^2)\leq \P_{x,i}(\tau_{B(x,\varepsilon r)}\leq c_1r^2)\leq 1/2,
$$
which implies (a).
Hence
$$
\E_{x,i}\tau_{B(x_0,r)}\geq c_1r^2 \P_{x,i}(\tau_{B(x_0,r)}> c_1r^2)\geq  c_1r^2/2.
$$
Then (b) follows.
\end{proof}

\medskip

For a measure $\mu$ on $\R^d$ and $y\in \R^d$,
we use $\mu (dx-y)$ to denote the measure $\nu (dx)$
defined by $\nu (A):=\mu (A-y)$ for $A\in {\cal B} (\R^d)$,
where $A-y:=\{x-y: x\in A\}$.
We know how the switched Markov processes jumps at the switched times between different
plates.
The following describes how the switched Markov process $(X_s, \Lambda_s)$ jumps at
non-switching times.

\begin{proposition}\label{P:4.4}
Assume conditions {\rm(A1)-(A3)} hold.
Suppose  $A$ and $B$ are two bounded open subsets of $\R^d$ having a positive distance apart and $i_0\in \M$.
Then
\begin{equation}\label{E:15}
\sum_{s\leq t}\indi_{\{X_{s-}\in A, X_s\in B, \Lambda_s=i_0\}}-\int_0^t \indi_A\(X_s\)
\indi_{\{i_0\}}\(\Lambda_s\)\pi_{\Lambda_s}(X_s,B-X_s) ds
\end{equation}
is a $\P_{x, i}$-martingale for each $(x, i)\in \R^d\times \M$.
\end{proposition}

\begin{proof}
  Let $A_1$ be a bounded open subset of $\R^d$ so that $\overline A \subset A_1\subset \overline A_1
  \subset B^c$.
  Let $v(\cdot,j)\equiv 0$ for all $j\not= i_0$, and $v(\cdot, i_0) \in C^2_b(\R^d)$ so that
  $v(x, i_0)=0$ on $A_1$ and $ v(x, i_0)=1$ on $B$. Fix $(x, i)\in \R^d\times \M$.
Note that
	$$M^v(t):=v(X_t,\Lambda_t)-v(X_0,\Lambda_0)-\int_0^t\mathcal G v(X_s,\Lambda_s)ds$$
	is a $\P_{x, i}$-martingale,  so is $\int_0^t \indi_A(X_{s-})dM^v(s)$.
Define $\tau_0=0$,  $\tau_1=\inf\{t\geq 0: X_t\in A \}$, $\tau_2 = \inf\{t\geq \tau_1: X_t \in A_1^c\}$,
and for $k\geq 2$,
$$
\tau_{2k-1}=\inf\{t\geq \tau_{2 (k-1)}: X_t\in A \}, \qquad
\tau_{2k} = \inf\{t\geq \tau_{2k-1}: X_t \in A_1^c\}.
$$
Note that $v(X_t, \Lambda_t)=0$ for $t\in \cup_{k\geq 1} [\tau_{2k-1}, \tau_{2k})$ and
$\indi_A (X_{t-})=0$ for $t\in \cup_{k\geq 1} [\tau_{2(k-1)}, \tau_{2k-1})$. Thus
the Riemann sum approximation of stochastic integral yields that
	\bea
	\ad \int_0^t \indi_A\(X_{s-}\)dM^v(s) \\
\ad= \sum_{k=1}^\infty  \indi_A (X_{ \tau_{2k}\wedge t-} ) \Big(
v(X_{\tau_{2k}\wedge t}, \Lambda_{\tau_{2k}\wedge t} ) -v (X_{\tau_{2k}\wedge t-}, \Lambda_{\tau_{2k}\wedge t-} ) \Big)  \\
	\ad \quad -
	\int_0^t \indi_A\(X_{s-} \)\mathcal{G}v(X_s, \Lambda_s)ds\\
	\ad = \sum_{s\le t}\indi_A\(X_{s-}\)\big[v(X_s, \Lambda_s)-v(X_{s-}, \Lambda_{s-}) \big]\\
	\ad \quad -
	\int_0^t \indi_A\(X_s\)\mathcal{G}v(X_s, \Lambda_s)ds.
	\eea
 Since $v(y, j)=0$ on $ A_1\times\mathcal{M}$, we have
	$$\mathcal Gv(y,j)=\int_{\R^d} v(y+z,j)\pi_j(y,dz) =\int_{\R^d} v(z,j)\pi_j(y, d z-y)$$
 for every $(y, j)\in A_1\times \mathcal{M}$.
	Therefore,
\bea \ad\sum_{s\le t}\indi_A\(X_{s-} \)\big[v(X_s, \Lambda_s)-v(X_{s-}, \Lambda_{s-})\big]\\
	\aad \quad -\int_0^t\indi_A\(X_s\)\int_{\R^d} v(z,\Lambda_s)\pi_{\Lambda_s}(X_s, dz-X_s) ds
	\ \hbox{
		is a $\P_{x, i}$-martingale.}\eea
	Because $A$ and $B$ are a positive distance from each other, the sum on the left of the above formula is in fact a finite sum.
	With these facts we can pass to the limit to conclude that
	\bea \ad \sum_{s\leq t}\indi_A\(X_{s-}\)\big[\indi_{B\times \{i_0\}}\(X_s, \Lambda_s\)-\indi_{B\times \{i_0\}}\(X_{s-}, \Lambda_{s-}\)\big]\\
	\aad \quad -\int_0^t\indi_A\(X_s\)\int_{\R^d}  \indi_{B\times \{i_0\}}\(z, \Lambda_s\)\pi_{\Lambda_s}(X_s, dz-X_s) ds \ \hbox{ is a $\P_{x, i}$- martingale,}\eea
which implies
	$$\sum_{s\leq t}\indi_{\{X_{s-}\in A, X_s\in B, \Lambda_s=i_0\}}-\int_0^t \indi_A\(X_s\) \indi_{\{ \Lambda (s)=i_0\}} \pi_{\Lambda_s}(X_s, B-X_s) ds
	$$
	is a $\P_{x, i}$-martingale.
\end{proof}

\begin{proposition}\label{L:4.5}
There exist $\wdt r_0\in (0,1/2]$ and $C_5>0$ such that for any $x_0\in\R^d$ and any $r\in(0,\wdt r_0)$, we have
\begin{equation}\label{E:16}
\sup_{(x,i)\in B(x_0,r)\times\M}\E_{x,i}\tau_{B(x_0,r)}\leq C_5r^2.
\end{equation}
\end{proposition}

\begin{proof}
Let $u(x)\in C^2(\R^d)$ be a convex function in $x$ with values in $[0,10]$ and increase with respect to $|x|$ such that
\begin{eqnarray*}
u(x)=
|x|^2,\quad |x|\leq 2.
\end{eqnarray*}
Let $\wdt r_0\in (0,1/2)$ be sufficiently small.
For $x_0\in\R^d$ and $r\in (0,\wdt  r_0)$, let $v(x,i)=u(\frac{x-x_0}{r})$. Then for any $(x, i)\in B(x_0,r)\times \M$, since $v(\cdot, \cdot)$ is bounded in $[0, 10]$ and $Q\cd$ is bounded, there exists $c_1>0$ such that
\beq{E:16.01}Q(x)v(x, \cdot)(i)\ge -c_1.\eeq
Moreover,
\beq{E:16.02}\barray
\LL^{(c)}v(x,i)\ad :=\sum^d_{k,l=1}a_{kl}(x,i)\frac{\p^2v(x,i)}{\p x_k\p x_l}+\sum^d_{k=1}b_{k}(x,i)\frac{\p v(x,i)}{\p x_k}\\
\ad =\sum^d_{k=1}2a_{kk}(x,i)r^{-2}+\sum^d_{k=1}2b_{k}(x,i)(x_k-x_{0, k})r^{-2}\\
\ad \ge c_2r^{-2}-c_3r^{-1}\geq c_4r^{-2},
\earray\eeq
provided $\wdt r_0$ is small enough.
Define
$$\LL^{(j)} v(x, i)=\int_{\R^d} \big[f(x+z,i)-f(x,i)-\nabla f(x,i)\cdot z \indi_{\{|z|<1\}}\big]\pi_i(x, d z) .$$
 We break $\LL^{(j)}   v(x, i)$ into two parts, $|z|\leq 1$ and
$|z|>1$, respectively. For the first part, by the convexity of  $u(x)$, we deduce
\beq{E:16.03}
\int_{|z|\leq1}[v(x+z,i)-v(x,i)-\nabla v(x,i)\cdot z]\pi_i(x, d z) \geq0.
\eeq
For the second part with $|z|>1$, since $r< \frac 1 2$, we know $x+z\notin B(x_0,r)$ for any $x\in B(x_0,r)$. Then
we have $v(x+z,i)\geq1$ and $v(x,i)\leq1$, it follows that
\beq{E:16.04}\int_{|z|>1}[v(x+z,i)-v(x,i)]\pi_i(x, d z)\geq0.\eeq
Since $\P_{x, i}$ solves the martingale problem, together with \eqref{E:16.01}, \eqref{E:16.02}, \eqref{E:16.03}, and \eqref{E:16.04}, we deduce that
\beq{E:16.1}\barray \ad  \E_{x,i}v\big(X_{t\wedge\tau_{B(x_0,r)})},\Lambda_{t\wedge\tau_{B(x_0,r)})}\big)-v(x,i)\\
\aad =\E_{x,i}\int_0^{t\wedge\tau_{B(x_0,r)}}\mathcal Gv(X_s,\Lambda_s)ds\geq c_5r^{-2}\E_{x,i}(t\wedge\tau_{B(x_0,r)}).
\earray\eeq
By the definition of $v(x,i)$,
\beq{E:16.2}\E_{x,i}v(X_{t\wedge\tau_{B(x_0,r)})},\Lambda_{t\wedge\tau_{B(x_0,r)})})-v(x,i)\leq 10.\eeq
It follows from \eqref{E:16.1} and \eqref{E:16.2} that $$c_5r^{-2}\E_{x,i}(t\wedge\tau_{B(x_0,r)})\le 10.$$
The conclusion follows by letting $t\to \infty$.
\end{proof}

In the remaining of this section, we assume that   conditions (A1)-(A4)  hold.

\begin{definition}{\rm
		The generator $\G$ or the matrix function $Q\cd$
		is said to be \textit{strictly irreducible} on $D$ if
		for any $i, j\in \M$ and $i\ne j$, there exists $q_{ij}^0>0$ such that $\inf\limits_{x\in D}q_{ij}(x)\ge q_{ij}^0$.
		
	}
\end{definition}

\begin{proposition}\label{P:E-H}
Assume conditions {\rm(A1)-(A4)} hold.
Let $x_0\in \R^d$ and $r\in (0, \wt r_0)$. Suppose that the operator $\G$ is strictly irreducible on $B(x_0, r)$. Let
 $H: \mathbb{R}^d\times {\cal M} \mapsto \mathbb{R}$ be a bounded non-negative function supported in $B(x_0, 2r)^c\times {\cal M}$.
Then there exists a constant $C_6>0$ such that for any $x, y\in B(x_0, r/2)$ and any $i\in \M$,
$$
\E_{x,i}H\(X_{\tau_{B(x_0,r)}}, \Lambda_{\tau_{B(x_0,r)}}\)\leq C_6\alpha_{2r}\E_{y,i}H\(X_{\tau_{B(x_0,r)}}, \Lambda_{\tau_{B(x_0,r)}}\).
$$
\end{proposition}

\begin{proof}
	Denote $B=B(x_0, r)$. Define $$u(x, i)=\E_{x,i}H\(X_{\tau_{B}}, \Lambda_{\tau_{B}}\) \quad \text{for}\quad (x, i)\in B\times \M.$$
	Since $H=0$ on $B(x_0,2r)\times {\cal M}$, we have by using the L\'evy system formula of $Y=(X, \Lambda)$ given by Proposition \ref{P:4.4}
	that
\beq{E:17.5}
\barray
	u(x, i) \ad= \E_{x,i} \left[ H \(X_{\tau_{B}}, \Lambda_{\tau_{B}}\);
	X_{\tau_{B}-} \in \overline{B}; X_{\tau_{B}}\in B(x_0, 2r)^c \right]
	\\
	\ad =\E_{x,i}\left[ \int_0^{\tau_{B}} \int_{B(x_0, 2r)^c} H\(z, \Lambda_s\) \wdt\pi_{\Lambda_s} \(X_s, z-X_s\)dz ds \right].
	\earray\eeq
	We deduce that
	\bea
		u(x, i) \ad
		  \le  \E_{x,i}\left[ \int_0^{\tau_{B}} \int_{B(x_0, 2r)^c}\sum_{j=1}^{m}  H(z, j)\wdt \pi_{j} \(X_s, z-X_s\)dz ds \right]\\
		  \ad    \le  \Big( \sum_{j=1}^{m}  \int_{B(x_0, 2r)^c}
		H(z, j)\sup\limits_{w\in B}\wdt\pi_j(w, z-w)dz \Big)
		\E_{x,i}  \tau_{B }  \\
		\ad \le  c_1 r^2 \sum_{j=1}^{m}  \int_{B(x_0, 2r)^c}
		H(z, j)\sup\limits_{w\in B}\wdt\pi_j(w, z-w)dz\\
		\ad = c_1r^2 M,
\eea
where the last inequality is
a consequence of
Proposition \ref{L:4.5}, and
$$M=\sum\limits_{j=1}^m M_j, \quad  M_j=\int_{B(x_0, 2r)^c}
H(z, j)\sup\limits_{w\in B}\wdt\pi_j(w, z-w)dz.$$
Thus,
\beq{E:17.7} u(x, i)\le c_1r^2M \quad \text{for }  (x, i)\in B\times \M.
\eeq
	Let $\tau_1=\inf\{t>  0: \Lambda_t\ne \Lambda_0\}$ and denote the Green operator of $\LL_i +q_{ii}$ in $B$ by $G^i_B$.
Define $$h_i(x)=\E_{x, i}[H \(X_{\tau_{B}}, \Lambda_{\tau_{B}}\);  \tau_B<\tau_1], \quad (x, i)\in B\times \M .$$
By the strong Markov property of $\(X_t, \Lambda_t\)$, we have
\beq{E:18.1}u(x, i)= h_i(x) + \sum\limits_{k\ne i}G^i_B\Big(q_{ik}(\cdot)u(\cdot, k)\Big)(x).
\eeq	
By \eqref{E:17.7} and the fact that $||q_{ik}||_{\infty}=\sup\limits_{x\in \R^d}|q_{ik}(x)|<\infty$, we arrive at
\begin{eqnarray*}
\sum\limits_{k\ne i}G^i_B\Big(q_{ik}(\cdot)u(\cdot, k)\Big)(x)&\le & \sum\limits_{k\ne i} c_1r^2 M ||q_{ik}||_{\infty} \E_{x, i}\tau_B\\
&\le &\sum\limits_{k\ne i} c_1r^2 M ||q_{ik}||_{\infty} c_1r^2=c_2Mr^4.
\end{eqnarray*}
	Combining above estimates, we obtain
	\beq{E:18.2}
	u(x, i)\le h_i(x) + c_2Mr^4 \quad \text{for} \quad (x, i)\in B\times \M.
	\eeq
	Next, we drive an lower bound for $\E_{x, i}(\tau_B\wedge \tau_1)$ for $(x, i)\in B(x_0, r/2)\times \M$.
	By Proposition \ref{L:4.3}, there exists $c_3>0$ such that
	$$\P_{x, i}\(\tau_B>c_3r^2\)\ge \dfrac{1}{2}.$$
	It follows that
	\bea
	\P_{x, i} (\tau_B\wedge\tau_1\ge c_3r^2 )\ad \ge
	\P_{x, i} \(\tau_B> c_3r^2 \text{ and } \tau_1\ge c_3r^2\)\\
	\ad \ge \exp\(-||q_{kk}||_\infty c_3r^2\)\P_{x, i}\(\tau_B>c_3r^2\)\\
	\ad
	\ge  \dfrac{1}{2}\exp\(-||q_{kk}||_\infty c_3 \wdt r_0^2\)=:c_4.
	\eea
	Then we obtain
	\beq{E:18.2.2}\E_{x, i}\(\tau_B\wedge \tau_1\)\ge c_3c_4r^2 \quad  \text{for } (x, i)\in B(x_0, r/2)\times \M.\eeq
	 By assumption (A4), for any $z\in B(x_0, 2r)^c$,
$$\sup\limits_{w\in B}\wdt\pi_j\(w, z-w\)\le \alpha_{2r}\inf\limits_{w\in B}\wdt\pi_j\(w, z-w\).$$
By this inequality and
\eqref{E:18.2.2},
we have by using L\'evy system formula of $Y=(X, \Lambda)$,
	\beq{E:18.3}
	\barray
 h_i(x)\ad= \E_{x,i}\left[ \int_0^{\tau_{B}\wedge \tau_1} \int_{B(x_0, 2r)^c} H\(z, \Lambda_s\) \wdt\pi_{\Lambda_s} \(X_s, z-X_s\)dz ds \right]\\
	\ad \ge \E_{x,i}\left[ \int_0^{\tau_{B}\wedge\tau_1} \int_{B(x_0, 2r)^c} H(z, i) \inf\limits_{w\in B}\wdt\pi_{i} \(w, z-w\)dz ds \right] \\
	\ad \ge \alpha_{2r}^{-1} M_i \E_{x, i}(\tau_B\wedge \tau_1)\\
	\ad \ge c_5\alpha_{2r}^{-1} M_i r^2 \quad \text{for}\quad x\in B(x_0, r/2).
	\earray
	\eeq
	On the other hand,
	\beq{E:18.3.2}
	\barray
	 h_i(x)\ad =\E_{x,i}\left[ \int_0^{\tau_{B}\wedge \tau_1} \int_{B(x_0, 2r)^c} H\(z, \Lambda_s\) \wdt\pi_{\Lambda_s} \(X_s, z-X_s\)dz ds \right]\\
	\ad \le \E_{x,i}\left[ \int_0^{\tau_{B}\wedge\tau_1} \int_{B(x_0, 2r)^c} H(z, i) \sup\limits_{w\in B}\wdt\pi_{i} \(w, z-w\)dz ds \right]\\
	\ad \le M_i \E_{x, i}\tau_B\\
	\ad \le c_6 M_i r^2 \quad \text{for}\quad x\in B(x_0, r/2).
	\earray
	\eeq
Note that for $i\ne k$, $\inf\limits_{x\in \R^d}q_{ik}(x)\ge q_{ik}^0>0$. By  \eqref{E:18.1} and \eqref{E:18.3}, we have
\beq{E:18.4}
\barray
\sum\limits_{k\ne i}G^i_B\Big(q_{ik}(\cdot)u(\cdot, k)\Big)(x) \ad \ge \sum\limits_{k\ne i} c_5\alpha_{2r}^{-1}M_kr^2 G^i_B\Big(q_{ik}(\cdot)\indi_{B(x_0, 3r/4)}(\cdot)\Big)(x)\\
\ad \ge \sum\limits_{k\ne i} c_5\alpha_{2r}^{-1} M_kr^2 q_{ik}^0  \E_{x, i}\( \tau_{B(x_0, 3r/4)}\wedge \tau_1\)\\
\ad \ge  \sum\limits_{k\ne i} c_5\alpha_{2r}^{-1} M_kr^2 q_{ik}^0 c_7 r^2\\
\ad = c_8 \alpha_{2r}^{-1}r^4 \sum\limits_{k\ne i} M_k \quad \text{for } x\in B(x_0, r/2).
\earray
\eeq
In the above, we used the fact that $\E_{x, i}\( \tau_{B(x_0, 3r/4)}\wedge \tau_1\)\ge c_7r^2$. This can be derived in the same way as that of \eqref{E:18.2.2}.
 By \eqref{E:18.1}, \eqref{E:18.2}, \eqref{E:18.3},  \eqref{E:18.3.2}, and  \eqref{E:18.4}, for any $x, y\in B(x_0, r/2)$ and $i\in \M$, we have
 \bea
 u(y, i) \ad = h_i(y) + \sum\limits_{k\ne i}G^i_B\Big(q_{ik}(\cdot)u(\cdot, k)\Big)(y)
\\
\ad \ge c_5\alpha_{2r}^{-1} M_i r^2 + c_8 \alpha_{2r}^{-1} r^4 \sum\limits_{k\ne i} M_k\\
\ad \ge c_9 \alpha_{2r}^{-1} \Big( M_i r^2 +  r^4 \sum\limits_{k\ne i} M_k\Big)\\
\ad \ge c_{10} \alpha_{2r}^{-1} \Big( M_i r^2 +  r^4 M\Big) \ge c_{11} \alpha_{2r}^{-1} u(x, i).
 \eea
The proof of the proposition is complete.
\end{proof}

\begin{theorem}\label{T:harnack}
Assume conditions {\rm(A1)-(A4)} hold.
Let $D\subset \R^d$ be a bounded connected open set
and $\K$ be a compact set in $D\subset \R^d$. Suppose that $\G$ is strictly irreducible on $D$. Then there exists $C_7>0$ which depends only on $D, \K$ and  operator $\G$ such that if $f(\cdot,\cdot)$ is a nonnegative, bounded function in $\R^d\times\M$
 that is $\mathcal G$-harmonic in $D\times\M$, we have
\begin{equation}\label{E:20}
f(x,i)\leq C_7f(y,j)  \quad \hbox{for } x, y\in \K \hbox{ and } i, j\in {\cal M}.
\end{equation}
\end{theorem}

\begin{proof}
We first show that for each fixed ball  $B(x_0,4R)\subset D$ with $R<\frac 1 8 \wedge \wt r_0$ (where $\wt r_0$ is given in Proposition \ref{L:4.5}), there exists a constant $C>0$ that depends only on $R$ and operator $\G$ such that for any nonnegative, bounded and $\mathcal G$-harmonic function $f(\cdot,\cdot)$ in $B(x_0,4R)\times\M$, we have
\begin{equation}\label{E:21}
f(x,i)\leq Cf(y,j) \quad \hbox{for } x, y\in B(x_0,R)
\hbox{ and } i, j\in {\cal M}.
\end{equation}
By looking at $f+\varepsilon$ and sending $\varepsilon$ to 0, we may suppose that $f$ is bounded below by a positive constant. By looking at $af(x,i_0)$ for a suitable constant $a$ if needed, we may
assume that $\inf\limits_{(x, i)\in B(x_0,R)\times \M}f(x,i)=1/2$.

\bigskip

\noindent (a) Let us recall several results. Let $r< \wt {r}_0<1/2$.
By Proposition \ref{P:4.6.2}, there exits a constant $c_1>0$ such that for any $x\in B(x_0, 3R/2)$ and any $i\in \M$,
\beq{E:21.2}
\P_{\overline{x}_i, i}\(T^{i}_{B(x, r/2)}<\tau_{B(x_0, 4R)}\)\ge c_1r^6.
\eeq
By Proposition \ref{P:4.6}, there exists a nondecreasing function $\Phi: (0, \infty)\mapsto (0, \infty)$ such that if  $A$ is a Borel subset of $B(x,r)$ and $|A|/r^d\ge \rho$ for a given $\rho$, then for any $(y, i)\in B(x,r)\times \M$ and $r\in (0, \wt r_0)$,
\begin{equation}\label{E:21.3}
\P_{y,i}(T^{i}_A<\tau_{B(x, 2r)})\geq \dfrac{1}{2}\Phi\(|A|/r^d\).
\end{equation}
By  Proposition \ref{P:E-H} and $H$
being a nonnegative function supported on $B(x, 2r)^c$,  for any $y,z\in B(x,r/2)$
and $i\in {\cal M}$,
\begin{equation}\label{E:22}
\E_{y,i}H(X_{\tau_{B(x,r)}}, \Lambda_{\tau_{B(x,r)}})\leq c_2\alpha_{2r}\E_{z,i}H(X_{\tau_{B(x,r)}}, \Lambda_{\tau_{B(x,r)}}).
\end{equation}
To proceed, we first consider the case that
\beq{E:21.1}
\inf\limits_{x\in B(x_0, 2R)}f(x, i)<1 \quad \text{for each} \quad i\in \M.
\eeq
Thus, there exists
$\{\overline{x}_i\}_{i\in \M}$ such that
 \beq{E:21.1.2}\overline{x}_i\in B(x_0, 2R) \quad \text{and}\quad f(\overline{x}_i, i)<1.\eeq

\noindent
(b)	
For $n\ge 1$, let
$$r_n=c_3R/n^2,$$
where $c_3$ is a positive number such that $\sum_{n=1}^\infty r_n< R/4$ and $r_n\in (0, \wt r_0)$ for all $n$, that is,
\beq{E:24}c_3< \dfrac{1}{4\sum_{n=1}^{\infty} 1/n^{2}} \wedge \dfrac{\wt r_0}{R}.\eeq
In particular, it implies $r_n< R/4$.  Let $\xi$, $c_4$, $c_5$ be positive constants to be chosen later. Once these constants have been chosen, we can take $N_1$ large enough so that
\beq{E:25}\xi N_1\exp(c_4n)c_5r_n^{6+\beta}\ge 2\kappa_2 \quad \text{for all} \quad n=1, 2, \dots
\eeq
The constants $\kappa_2$ and $\beta$ are taken from assumption (A4).
Such a choice is possible since $c_4>0$ and $r_n=c_3R/n^2$.
Suppose that there exists $(x_1, i_1)\in B(x_0,R)\times \M$ with $f(x_1,i_1)=N_1$ for $N_1$ chosen above.
We will show that in this case there exists a sequence $\{(x_k, i_k): k\ge 1\}$ with
\begin{equation}\label{E:26}
\barray
\aad (x_{k+1}, i_{k+1})\in B(x_{k},2r_{k})\times {\cal M}\subset B(x_0,3R/2) \times {\cal M},\\
\aad N_{k+1}:=f(x_{k+1}, i_{k+1})\ge N_1\exp\(c_4(k+1)\).\earray\end{equation}
\bigskip

\noindent
(c) Suppose that we already have $\{(x_k, i_k); 1\leq k \leq n\}$ so  that \eqref{E:26} is satisfied for $k=1, \dots, n-1.$
Define
 $$A_n=\Big\{y\in B(x_n,r_n/2): f(y,i_n)\geq
\dfrac{\xi N_n r_n^\beta}{\kappa_2}
 \Big\}.$$	
We claim that
\begin{equation}\label{E:27}
\frac{|A_n|}{|B(x_n,r_n/2)|} \leq \frac 1 4.
\end{equation}
Suppose on the contrary,
$\frac{|A_n|}{|B(x_n,r_n/2)|}> 1/ 4$. Let $F$ be a compact subset of $A_n$ such that $\frac{|F|}{|B(x_n,r_n/2)|}> 1/4$. Then $F\subset B(x_0,2R)$.
By \eqref{E:21.2},
$$\P_{\overline{x}_{i_n}, i_n}\(T^{i_n}_{B(x_n, r_n/2)}<\tau_{B(x_0, 4R)}\)\ge c_1r_n^6,$$
where $c_1$ is independent of  $x_n$ and $r_n$. By the strong Markov property of $\(X_t, \Lambda_t\)$, we have
\bea
\P_{\overline{x}_{i_n}, i_n} \( T_F^{i_n}<\tau_{B(x_0, 4R)} \)\ad \ge \E_{\overline{x}_{i_n}, i_n} \Big[ \P_{X_{T^{i_n}_{B(x_n, r_n/2)}}, i_n}\(T_F^{i_n}<\tau_{B(x_n, r_n)}\); T^{i_n}_{B(x_n, r_n/2)}<\tau_{B(x_0, 4R)}         \Big]\\
\ad \ge \dfrac{1}{2}\Phi\(\dfrac{2^d|F|}{r_n^d}\)
\P_{\overline{x}_{i_n}, i_n}\(T^{i_n}_{B(x_n, r_n/2)}<\tau_{B(x_0, 4R)}\)\\
\ad \ge \dfrac{1}{2}\Phi\(\dfrac{\al(d)}{4}\)c_1r_n^6,
\eea
where $\al(d)$
is the volume of the unit ball in $\R^d$.

We take $c_5=\dfrac{1}{2}\Phi\(\al(d)/4\)c_1$. By the definition of $\mathcal{G}$-harmonicity and the above estimates, we obtain
\beq{E:27.2}\barray
1 > f(\overline{x}_{i_n},i_n) \ad \geq \E_{\overline{x}_{i_n},i_n}[f(X_{T^{i_n}_F\wedge\tau_{B(x_0,4R)}},\Lambda_{T^{i_n}_F\wedge\tau_{B(x_0,4R)}}); T^{i_n}_F<\tau_{B(x_0,4R)}]\\
\ad \ge  \dfrac{\xi N_n r_n^\beta}{\kappa_2}
\P_{\overline{x}_{i_n},i_n}(T^{i_n}_F<\tau_{B(x_0,4R)})\\
\ad \ge
\dfrac{\xi N_n r_n^{\beta+6} c_5}{\kappa_2}

\\
\ad \ge
2,
\earray\eeq	
which is a contradiction. Note that the last inequality follows from $N_n\ge N_1\exp\(c_3n\)$
and our choice of $N_1$ given by \eqref{E:25}.
  Thus, \eqref{E:27} is valid. Therefore, there is a compact subset $\wdt{F}$ of $B(x_n,r_n/2)\setminus A_{n}$ such that $|\wdt{F}|\geq\frac 1 2|B(x_n,r_n/2)|$. By the definition of $\wdt{F}$ and $A_{n}$,
$$f(x,i_n)< \dfrac{\xi N_n r_n^\beta}{\kappa_2}\quad \text{for}\quad  x\in \wdt{F}.$$
Denote $\tau_{r_n}:=\tau_{B(x_n,r_n)}$,   $p_n:=\P_{x_n,i_n}(T^{i_n}_{\wdt{F}}<\tau_{r_n})$ and $M_n:=\sup_{(y, j)\in B(x_n,2r_n)\times {\cal M} }f(y, j)$. Since $|\wdt{F}|\geq\frac 1 2|B(x_n,r_n/2)|$, using \eqref{E:21.3}, we obtain \beq{E:27.3}p_n\ge \dfrac{1}{2}\Phi\(\dfrac{\al(d)}{2^{d+1}}\):= c_6 \quad \text{for}\quad n=1, 2, \dots\eeq
	By the definition of $\G$-harmonic function and the right continuity of the sample paths of $\(X_t, \Lambda_t\)$, we have
\beq{E:28}\barray
N_n\ad  =  f(x_n,i_n)=\E_{x_n,i_n}[f(X_{T^{i_n}_{\wdt{F}}},\Lambda_{T^{i_n}_{\wdt{F}}}):T^{i_n}_{\wdt{F}}<\tau_{r_n}]\\
\ad \quad +\E_{x_n,i_n}[f(X_{\tau_{r_n}},\Lambda_{\tau_{r_n}}):X\(\tau_{r_n}\)\in B(x_n,2r_n),\tau_{r_n}<T^{i_n}_{\wdt{F}}] \\
\ad \quad +\E_{x_n,i_n}[f(X_{\tau_{r_n}},\Lambda_{\tau_{r_n}}):X\(\tau_{r_n}\)\notin B(x_n,2r_n),\tau_{r_n}<T^{i_n}_{\wdt{F}}] \\
\ad \le \dfrac{\xi N_n r_n^\beta}{\kappa_2}+M_n(1-p_n) \\
\ad \quad +\E_{x_n,i_n}[f(X_{\tau_{r_n}}, \Lambda_{\tau_{r_n}}):X_{\tau_{r_n}}\notin B(x_n,2r_n),\tau_{r_n}<T^{i_n}_{\wdt{F}}].
\earray\eeq
Take a point $y_n\in \wt F$. Then $f(y_n,i_n)< \dfrac{\xi N_n r_n^\beta}{\kappa_2}$.  We then deduce from \eqref{E:22} that
\begin{eqnarray*}
	\dfrac{\xi N_n r_n^\beta}{\kappa_2} &>& f(y_n,i_n)\\
	&\geq&\E_{y_n,i_n}[f(X_{\tau_{r_n}},\Lambda_{\tau_{r_n}}):X_{\tau_{r_n}}\notin B(x_n,2r_n)]\\
	&\geq&\frac{1}{c_2\alpha_{2r_n}}\E_{x_n,i_n}[f(X_{\tau_{r_n}},\Lambda_{\tau_{r_n}}):X_{\tau_{r_n}}\notin B(x_n,2r_n)].
\end{eqnarray*}
It follows that
\bea\E_{x_n,i_n}[f(X_{\tau_{r_n}},\Lambda_{\tau_{r_n}}):X_{\tau_{r_n}}\notin B(x_n,2r_n)]\ad \leq
\dfrac{\xi N_n r_n^\beta c_2\alpha_{2r_n}}{\kappa_2}\\
\ad
\le \dfrac{\xi c_2}{2^\beta}N_n ,\eea
where the last inequality is obtained by noting that $\alpha_{2r_n}\le \kappa_2 (2r_n)^{-\beta}$.
Hence by \eqref{E:28},
\beq{E:28.0}N_n\leq \Big(\dfrac{\xi }{\kappa_2}+\dfrac{\xi c_2}{2^\beta}\Big)N_n+M_n(1-p_n).\eeq
Denote $\eta=1-\Big(\dfrac{\xi }{\kappa_2}+\dfrac{\xi c_2}{2^\beta}\Big)$. Let $\xi>0$ be sufficiently small such that
$\dfrac{\eta}{1-c_6}>3/2.$ By \eqref{E:27.3} and \eqref{E:28.0}, $M_n/N_n> 3/2.$
Using the definition of $M_n$,
there is $(x_{n+1}, i_{n+1})\in B(x_n,2r_n)\times {\cal M}$ so that
\bea
N_{n+1}:=f(x_{n+1},i_{n+1})\ad \geq  3N_n/2.\eea
We take $c_4=\ln(3/2)$. Then \eqref{E:26} holds for $k=n$.
By induction, we have constructed a sequence of points $\{(x_k, i_k)\}$ such that \eqref{E:26} holds for all $k\ge 1$.
 It can be seen that $N_k\to \infty$ as $k\to \infty$, a contradiction to the assumption that $f$ is bounded.

Thus, for a positive constant $N_1$ sufficiently large such that \eqref{E:25} holds, we have $$f(x, i)< N_1 \quad \text{for all} \quad (x, i)\in  B(x_0, R)\times {\cal M}.$$ Since $\inf\limits_{(x, i)\in B(x_0,R)\times \M}f(x,i)=1/2$, we arrive at
\begin{equation}\label{E:29}
f(x,i )\leq 2 N_1f(y, j), \quad  x,y\in B(x_0,R)
\hbox{ and } i, j\in {\cal M}.
\end{equation}
For any compact set $\K\subset D$, we use a standard finite ball covering argument. Since $\K$ is compact,
there exists a finite number of points $z_k\in \K$, $k=1, 2, \dots ,n$ such that $$\K\subset\bigcup\limits_{k=1}^n B(z_i,R)\subset D,$$ and $|z_k-z_{k-1}|<R/2$. Let $x,y\in \K$ and $i, j\in {\cal M}$. Applying Harnack inequality \eqref{E:29} at most $n+1$ times, we obtain
$f(x,i )\leq (2 N_1)^{n+1}f(y, j)$.
	
Now we suppose that \eqref{E:21.1} is invalid. Then there exists $i\in \M$ such that $f(x, i)\ge 1$ for all $x\in B(x_0, 2R)$. Set $$K_i:=\inf\limits_{x\in B(x_0, 2R)}f(x, i), \quad K: =\sup\limits_{i\in \M}K_i, \quad  g(x, i):=f(x, i)/3K.$$
It follows that \eqref{E:21.1} holds with $g$ in place of $f$. Moreover, if $i_0\in \M$ and $K=K_{i_0}$, then  $$g(x, i_0)\ge 1/3 \quad \text{for}\quad x\in B(x_0, 2R).$$
By the same argument as in Theorem \ref{T:3.8}, there is a constant $c_7>0$ such that
\bea
g(x, i)\ad \ge G^i_{B(x_0, 2R)}\Big(q_{ii_0}(\cdot)g(\cdot, i_0)\Big)(x)\\
\ad \ge  c_7,
\eea
for all $(x, i)\in B(x_0, R)\times \M$, where $G^i_{B(x_0, 2R)}$ is the Green operator of $\wt X^i$ in $B(x_0, 2R)$. Note also that $\inf\limits_{x\in D}q_{ii_0}(x)\ge q_{ii_0}^0>0$.
The Harnack inequalities for $g$, and
for $f$
can be established
similarly
as in the previous case.
\end{proof}

\section{Further Remarks}
This paper has been devoted to switching jump diffusions.
Important properties such as maximum principle and Harnack inequality have been obtained. The utility and applications of these results will be given in a subsequent paper \cite{CCTY17} for obtaining recurrence and ergodicity of switching jump diffusions. The ergodicity can be used in a wide variety of control and optimization problems with  average cost per unit time objective functions (see also various variants of the long-run average cost problems in \cite{JY13}), in which the instantaneous measures are replaced by the corresponding ergodic measures.

We note that the references \cite{CK1, CK2, CK3, CWX, NT}
are devoted to regularity  for the parabolic functions
of non-local operators (on each of the parallel plane). The results obtained in this paper should be useful when
one considers
regularity  of the coupled systems or switched jump-diffusions.


\begin{thebibliography}{99}
\setlength{\baselineskip}{0.18in}

\bibitem{AGM99}
 Arapostathis, A.,  Ghosh, M.K., and  Marcus, S.I. (1999), Harnack's inequality
  for cooperative weakly coupled elliptic systems, {\it Comm. Partial Differential
  Equations}, {\bf 24}, 1555--1571.

\bibitem{AR16} Athreya, S. and Ramachandran, K. (2016) Harnack inequality for non-local Schr\"odinger operators, preprint available at
arXiv:1507.07289v5 [math.PR].


\bibitem{BK05}
 Bass, R.F.  and Kassmann, M. (2005) Harnack inequalities for non-local
  operators of variable order, {\it Trans. Amer. Math. Soc.}, {\bf 357},
 837--850.

\bibitem{BKK}
 Bass, R.F., Kassmann, M. and Kumagai, T. (2010) Symmetric jump processes: Localization,
heat kernels and convergence, {\it Ann. Inst. Henri Poincar\'e Probab. Stat.}, {\bf 46},
59--71.


\bibitem{BL02}
 Bass, R.F. and Levin, D.A. (2002)
 Harnack inequalities for jump
  processes, {\it Potential Anal.}, {\bf 17}, 375--388.

\bibitem{CS09} Caffarelli, L. and  Silvestre, L. (2009)
Regularity theory for fully nonlinear
integro-differential equations,
{\it Comm. Pure Appl. Math.} {\bf 62}, 597--638.

\bibitem{CCTY17} X. Chen, Z.-Q. Chen, K. Tran, and G. Yin, Recurrence and ergodicity for
a class of regime-switching jump diffusions, preprint, 2017.

\bibitem{CK1} Chen, Z.-Q. and Kumagai, T. (2003)
 Heat kernel estimates for stable-like processes
on $d$-sets.  {\it Stochastic Process Appl. \bf 108},
27-62.

\bibitem{CK2} Chen, Z.-Q. and Kumagai, T. (2008)
 Heat kernel estimates for jump processes of mixed types on
metric measure spaces.    {\it Probab. Theory Relat. Fields \bf 140}, 277-317.

\bibitem{CK3}
 Chen, Z.-Q. and Kumagai, T. (2010) A priori h\"older estimate, parabolic
  harnack principle and heat kernel estimates for diffusions with jumps,
  {\it Revista Matematica Iberoamericana}, {\bf 26}, 551--589.


\bibitem{C16}
 Chen,  Z.-Q.,  Hu, E.,  Xie, L. and  Zhang, X. (2016) Heat kernels for non-symmetric diffusions operators with jumps,
preprint.

\bibitem{CWX}
 Chen, Z.-Q., Wang, H. and Xiong, J. (2012) Interacting superprocesses with
  discontinuous spatial motion, {\it Forum Math.}, {\bf 24}, 1183--1223.

\bibitem{CZ}
  Chen, Z.-Q. and Zhao, Z. (1996) Potential theory for elliptic systems, {\it Ann.
  Probab.}, {\bf 24}, 293--319.

\bibitem{CZ97}
\leavevmode\vrule height 2pt depth -1.6pt width 23pt, (1997) Harnack principle
  for weakly coupled elliptic systems, {\it J. Differential Equations}, {\bf 139},
  261--282.

\bibitem{Evans10}
 Evans, L.C. (2010) {\it Partial differential equations},
 Amer. Math. Soc., Providence, RI,
  Second Ed.

\bibitem{Foondun09}
 Foondun, M. (2009) Harmonic functions for a class of integro-differential
  operators, {\it Potential Anal.}, {\bf 31}, 21--44.

\bibitem{INW}
 Ikeda, N., Nagasawa, M. and Watanabe, S. (1966) A construction of markov
  process by piecing out, {\it Proc. Japan Academy}, {\bf 42}, 370--375.

\bibitem{JY13}
Jasso-Fuentes, H.,  Yin, G., (2013)
{\it Advanced Criteria for Controlled Markov-Modulated Diffusions in an Infinite Horizon: Overtaking, Bias, and Blackwell Optimality}, Science Press, Beijing, China.

\bibitem{Kry87}
 Krylov, N.V. (1987) {\it Nonlinear Elliptic and Parabolic Equations of the
  Second Order},   D.
  Reidel Publishing Co., Dordrecht,
Translated from the Russian by P.L. Buzytsky.

\bibitem{Komatsu}
  Komatsu, T. (1973) Markov processes associated with certain integro-differential, {\it Osaka J. Math.}, {\bf 10}, 271-303.


\bibitem{Kushner90}
 Kushner, H. (1990) {\it Weak Convergence Methods and Singularly Perturbed
  Stochastic Control and Filtering Problems}, Birkh\"auser, Boston.

\bibitem{Liu16}
Liu, R. (2016) Optimal stopping of switching diffusions with state
  dependent switching rates,
{\it Stochastics}, {\bf 88},
586–-605.

\bibitem{MY06}
 Mao, X. and Yuan, C. (2006) {\it Stochastic Differential Equations with Markovian
  Switching}, Imperial College Press, London, UK.

\bibitem{M}
 Meyer, P. (1975) Renaissance, recollements, m\'elanges, relentissement de
  processus de markov, {\it Ann. Inst. Fourier}, {\bf 25}, 465--497.


\bibitem{MP} Mikulevicius, R. and  Pragarauskas, H. (1988)
On H\"older continuity of solutions of certain integro-differential equations.
{\it Ann. Acad. Scien. Fenn. Ser. A. I. Math.}, {\bf 13}, 231-238.

\bibitem{NT}
 Negoro, A. and Tsuchiya, M. (1989) Stochastic processes and semigroups
  associate with degenerate L\'evy generating operators, {\it Stochastics
  Stochastic Reports}, {\bf 26}, 29--61.

\bibitem{PW67}
  Protter, M.H. and  Weinberger, H.F. (1967) {\it Maximum Principles in
  Differential Equations}, Prentice-Hall, Inc., Englewood Cliffs, N.J.

\bibitem{Sharpe} Sharpe, M. (1986) {\it General Theory of Markov Processes}, Academic, New York, 1986.

\bibitem{SV05}
 Song, R. and Vondra{\v{c}}ek, Z. (2005) Harnack inequality for some
  discontinuous Markov processes with a diffusion part,
  {\it Glas. Mat. Ser. III},
 {\bf 40(60)}, 177--187.


\bibitem{Wang14}
 Wang, J. (2014) Martingale problems for switched processes, {\it Mathematische
  Nachrichten}, {\bf 287}, 1186--1201.

\bibitem{Xi09}
 Xi, F. (2009) Asymptotic properties of jump-diffusion processes with
  state-dependent switching, {\it Stochastic Process. Appl.}, {\bf 119},
  2198--2221.

\bibitem{YZ10}
 Yin, G. and Zhu, C. (2010) {\it Hybrid Switching Diffusions:  Properties and Applications},
 Springer, New York.


\end{thebibliography}
\end{document}